\documentclass[11pt,thmsa]{article}
\usepackage{amsthm, amsmath, amssymb, epsfig, multicol,enumerate,mathrsfs,appendix,arcs,float}
\usepackage[left=2.5cm,top=2cm,bottom=2.5cm,right=2.5cm]{geometry}
\newtheorem{teorema}{Theorem}[section]

\newtheorem{corollario}[teorema]{Corollary}
\theoremstyle{remark}

\begin{document}

\title{The hyperbola rectification from Maclaurin to Landen and the Lagrange-Legendre transformation for the elliptic integrals}
\author{Giovanni Mingari Scarpello,\,Daniele Ritelli and Aldo Scimone}
\date{}
\maketitle

\begin{abstract}
This article describes the main mathematical researches performed, in England and in the Continent between 1742-1827, on the subject of hyperbola rectification, thereby adding some of our contributions.
We start with the Maclaurin inventions on Calculus and their remarkable role in the early mid 1700s; next we focus a bit on his evaluation, 1742, of the \textit {hyperbolic excess}, explaining the true motivation behind his research. To his geometrical-analytical treatment we attach ours, a purely analytical alternative.
Our hyperbola inquiry is then switched to John Landen, an amateur mathematician, who probably was writing more to fix his priorities than to explain his remarkable findings. We follow him in the obscure proofs of his theorem on hyperbola rectification, explaining the links to Maclaurin and so on. With a chain of geometrical constructions, we attach our interpretation to Landen's treatment. Our modern analytical proof to his hyperbolic limit excess, by means of elliptic integrals of the first and second kind is also provided, and we demonstrate why the so called Landen transformation for the elliptic integrals cannot be ascribed to him. Next, the subject leaves England for the Continent: the character of Lagrange is introduced, even if our interest concerns only his 1785 memoir on irrational integrals, where the \textit {Arithmetic Geometric Mean}, AGM, is established by him. Nevertheless, our objective is not the AGM, but to detect the real source of the so-called Landen transformation for elliptic integrals. In fact, Lagrange's paper encloses a 
differential identity stemming from the AGM: integrating it, we show how it could be possible to arrive at the well-known Legendre recursive computation of a first kind elliptic integral, which appeared in his \textit{Trait\'{e}}, 1827, much after the Lagrange's paper. 

\end{abstract}

\section*{Introduction}
The reader should know we have been driven by three criteria, the first of which is to get him in touch with old Masters still capable of teaching many things, nowadays forgotten after 250 years, surmounting the main difficulties of language, notations, and often, of a completely different view of approaching questions.

The second criterion is our philological approach: thanks to the power of the Internet, we were able to keep close reference to the original prints of each work. As a matter of fact, for instance with John Landen, we realized how the difficult access to antique texts produced many mistakes and errors in the past due to inaccurate assertions being repeated without cross-checks, which then propagated over time.

The third step concerns our attached analytical research. Our Authors often write in involved, rather obscure ways and prove theorems in special cases only,  even if they have a greater generality. Thus we decided to work on some statements, not only for the sake of an independent check, but also to frame an old conquest within the context of a modern outline and language. In such a way we pass by means of modern mathematical objects (special functions, elliptic integrals, successive theorems) from philology to analytical inquiries, often through the computer algebra tools: modern tools and classical mind. By comparison of our computations to original results, we highlight the quality displayed at time by the Authors, who worked with plain old tools and relied solely on their brilliance and ingenuity. To Maclaurin's research we added some explanations and figures providing also a modern alternative formulation to the hyperbolic excess.
 Landen's writing is usually obscure, and his theorems are proved in special cases only, so for his excess computations, we provide alternative proofs. 
For Lagrange we show how his AGM transformation can lead to the  famous modular transformation for the elliptic integrals, mistakenly attributed to Landen. 

\section{Maclaurin} 

Maclaurin's name remains in the history of science thanks to the first textbook dealing with Newtonian Calculus: his \textit{Treatise of Fluxions}, Edinburgh, 1742. In it Maclaurin tried to present the Calculus with the \lq\lq rigour of the ancients'' supporting Newton in the polemics on Calculus fundaments' tenability raised in \textit {The Analyst, a discourse addressed to an Infidel Mathematician} \cite{Berk} by George Berkeley (1685-1753), who had criticized the foundations of the analysis in the works by Newton.\footnote{The \lq\lq infidel mathematician'' is believed to have been (perhaps) the astronomer  E. Halley, responsible for financing in 1687 the print of \textit{Principia} by Newton.}

\subsection{Maclaurin's  works and his \textit{Fluxions}}
Colin Maclaurin (1698-1746) lived and worked in a remarkable scientific context. His name is commonly associated with the \textit{Maclaurin series} $ f(x)=f(0)+f'(0)x + 1/2f''(0)x^{2} +\cdots$ despite the fact that it had been published by the Englishman Brook Taylor (1685-1731) in its more general form (Taylor series) already in 1715, and James Gregory (1638-1675) had used it in special cases. Maclaurin's papers on journals can be divided as issued in the Philosophical Transactions (9 articles between 1718 and 1743 concerning curves, equations, and \ldots \textit{Cells wherein the Bees deposit their Honey}).
A second group has been added to these writings, with two further publications about astronomy and in the Physical and Literary Society, Edinburgh, Vol. I.

Maclaurin's books covered each branch of Mathematics: \textit{Geometria Organica} London, 1720; \textit{A Treatise of Fluxions}, 2 volumes, Edinburgh, 1742; \textit{A Treatise of Algebra}, with an Appendix, \textit{De Linearum Geometricarum Proprietatibus generalibus}, 1748\footnote{A milestone theorem of the planar cubics theory (probably already known to Gua de Malves, 1740), is held in this Appendix:  if a straight line meets 2 real inflection points of a cubic, it will cross it again in a third real inflection point.}; {\it An Account of Sir Isaac Newton's Philosophy} (1748).
 In his \textit{Geometria} Maclaurin dealt with conics, cubics, quartics, and general properties of curves,
such as the famous \textit{trisectrix}, he met while studying the ancient problem of the angle's trisection and whose equation in cartesian coordinates is  $y^{2}=(3+x)x^{2}/({1-x}).$ The treatise \textit{Fluxions}, \cite{Fluxions}, had its origin in the Maclaurin's defense of the Newtonian doctrine of fluxions, expanding, well beyond a pamphlet, to more than 760 pages: a very complex construction where most of the time mathematics is presented, organized and applied to several physics problems. It therefore will not be described here, and we will stick only to a very superficial discussion of it. The reader can have a satisfactory idea of its internal partitions and contents by referring to Sageng's pages, \cite{Sageng}.

The n. 927 of  \textit{Fluxions} entitled: \textit {The construction of the elastic curve, and of other figures, by the rectification of the conic sections}, starts with:
\begin{quote}
The celebrated author who first resolved this as well as several other curious problems, after his account of this figure (which is commonly called the clastic curve), 
adds: \lq\lq Ob graves causas suspicor curvae nostrae constructionem a nullius sectionis conicae seu quadratura seu rectificatione pendere'' Act. Lips. 1694, page 272. But it is constructed by the rectification of equilateral hyperbola.
\end{quote}
The unmentioned author is Jakob Bernoulli, and such a motivation is by Maclaurin postponed to the theoretical treatment 
of hyperbola rectification, n. 805. We will concentrate on this treatment, and proceed to provide more elements on its employ in the elastica problem.

Maclaurin's extensive interest in almost all Mathematical Physics and Calculus of his time, led him to the problem of fluents, and, not only to solve the elastica, but to rectify the curves as well.
 In \textit{Fluxions}, n. 755, Maclaurin defines a research program concerning the classification of irrational fluents, to be followed by D'Alembert in his \textit{Recherches sur le calcul integral}, 1746, published in 1748 by the Berlin Academy, see \cite{Dal}. Whilst D'Alembert used a purely analytical approach without any figure, doing only algebraic computations by means of several changes of variable, Maclaurin, on the contrary performed the integration of irrational differentials by means of arcs of conics and often with the help of geometrical arguments. In such a way the influence of Maclaurin induced D'Alembert to study by means of an algebraic process (\textit{Recherches}, page 203) the fluent of
$$ 
{\rm d}x\sqrt{\frac{x}{x^{2}\pm fx+b^{2}}}, 
$$
establishing that it can be reduced to the addition of an arc of hyperbola of certain semiaxes plus an algebraic term.
At n. 798 the Maclaurin classification is a bit more strict: let us give an account of it.

\noindent First class: when a fluent can be represented in a finite number of algebraic terms, like the fluent of 
$$ 
\frac{{\rm d}x}{\sqrt{1\pm x}}.
$$
Second class: includes fluents like 
 $$ 
\frac{{\rm d}x}{\sqrt{1\pm x^{2}}},
$$
which can be reduced to areas of a circle and the hyperbolic areas of logarithms: they cannot be assigned in algebraic terms, but have been computed by several methods

Third class: fluents like
$$\sqrt{\frac{x}{1\pm x^{2}}}\,{\rm d}x\quad\text{or}\quad\frac{{\rm d}x}{\sqrt{x}\sqrt{1\pm x}}$$
which cannot be reduced to any form and are required in some useful problems. They can only be assigned by hyperbolic and elliptic arcs; namely the computation of the length of a hyperbola leads to fluents of this type. 
Maclaurin realized that the elastica analysis could lead to the same (hopeless) integrals met when trying to rectify the conical sections. In such a way he judged  the problem as solved, since its solution is given by a known and traceable curve; the practical computation was accomplished by expanding the function and performing a termwise integration.

\subsection{Inside the \textit{Fluxions}: the hyperbola and its excess}
 Before entering Maclaurin's topic of our interest, we need to present shortly some definitions which precede a couple of theorems due to Apollonius, and quoted by Maclaurin himself.
\subsubsection{Apollonius' theorems on hyperbola}
First of all, some terms will be recalled to be used later. Let us consider  the hyperbola $AP\ldots$, and $A'P'\ldots$ of equation
\begin{equation}
\frac{x^{2}}{a^{2}}-\frac{y^{2}}{b^{2}}=1. \label{ipxya}
\end{equation}
$a$ and $b$ being the \textit{semiaxes}; the focal axis FF', with the focal distance given by $\sqrt{a^{2}+b^{2}}$, named \textit{transverse} is assumed as $Ox$, while $Oy$ is named \textit{not transverse}, or \textit{imaginary}; the hyperbola asymptotes are given by $y=\pm\,(b/a)\,x$. We will take into account also the hyperbola:
\begin {equation}
\frac{y^{2}}{b^{2}}-\frac{x^{2}}{a^{2}}=1 \label{ipyx}
\end {equation}
having the same asymptotes and focal distance as \eqref{ipxya}, but exchanging transverse and not transverse axes. The hyperbol\ae\, of equations \eqref{ipxya} and \eqref{ipyx} are named \textit{conjugate}, and a whichever straight line passing through their common center O is said to be the \textit{diameter} for both. Two diameters (for instance $PP'$ and $P_{1}P_{1}'$, see \figurename~\ref{f01}) are said to be \textit{conjugate} whenever the tangents to hyperbola at each extreme  of one of them are parallel to the other one.
\begin{figure}[th]
\begin{center}
\scalebox{0.6}{\includegraphics{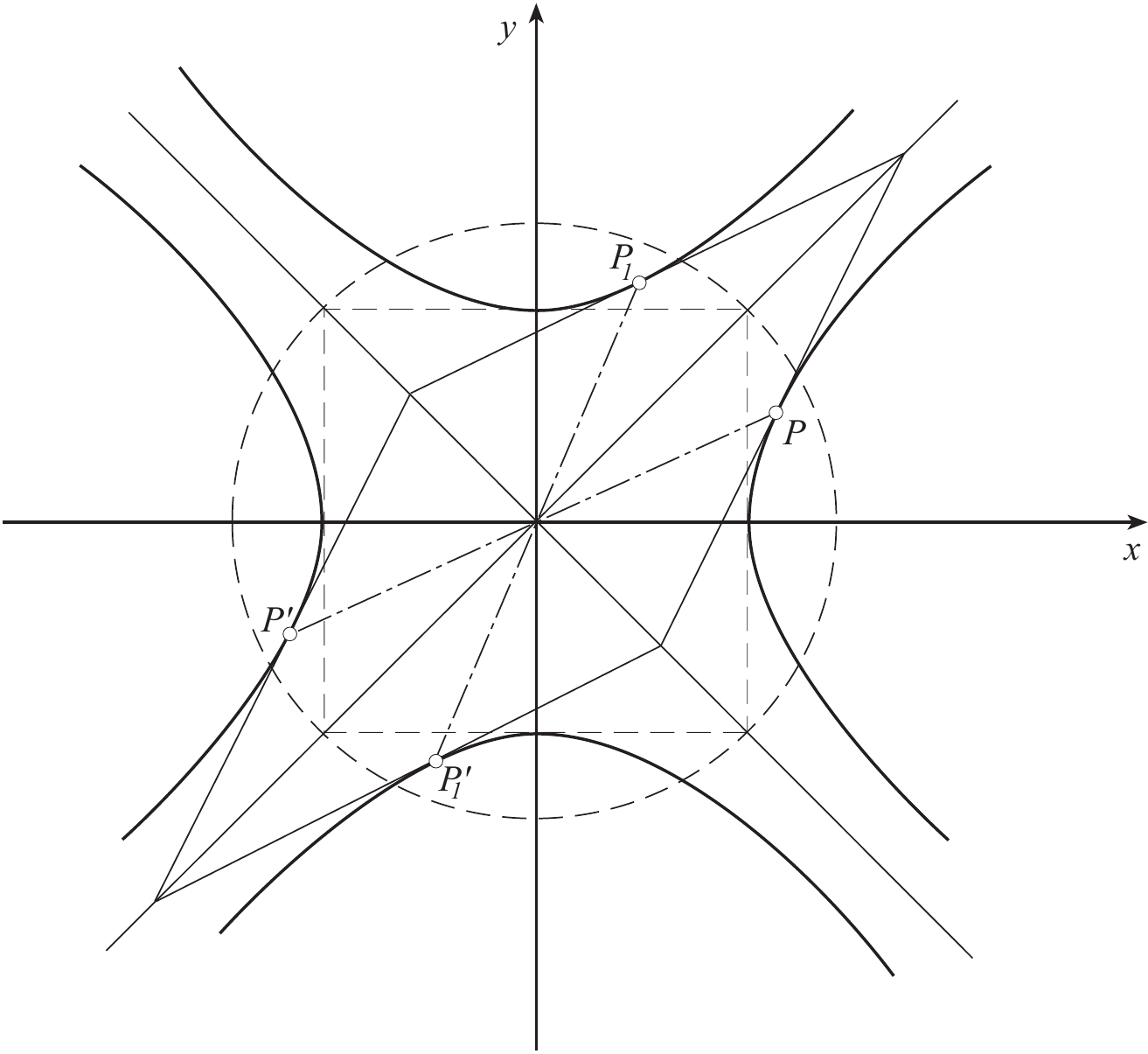}}\end{center}
\caption{The parallelogram relevant to a couple of conjugate hyperbol\ae}\label{f01}
\end{figure}
 In such a way the four straight lines touching a couple of conjugate diameters at the extremes define a parallelogram (marked as  $R_{1}WLG$ in \figurename~\ref{f02}) which circumscribes the couple of conjugate hyperbol\ae.

Now let us introduce Apollonius' theorems.
\begin{teorema}
For any hyperbola the absolute value of the difference of the squares of any couple of conjugate diameters has a constant value, given by $\vert{a^{2}-b^{2}}\vert.$
\end{teorema}
\begin{teorema}
The area of a parallelogram circumscribing two conjugate hyperbol\ae\, of equations \eqref{ipxya}, \eqref{ipyx}, having  the asymptotes as diagonals, is constant. Its value is given by $4ab$, namely the area of a rectangle whose sides are twice the semiaxes.
\end{teorema}
We make a special use of a corollary of this second theorem, corollary implicitly assumed by Maclaurin in his work on rectification. Minding \figurename~\ref{f03} we have:
\begin{corollario}
Given a couple of conjugate hyperbol\ae\,  of center S, for whichever semidiameter $SH$, the product of its length to the distance $\overline{SP}=p$ is then constant, being P  the  point where the perpendicular drawn from $S$ crosses the  tangent to the hyperbola parallel to $SH$:
\begin{equation}
\overline{SH}\times p=ab \label{scorollary}
\end{equation}
\end{corollario}
\begin{proof}
We refer to \figurename~\ref{f02}. 
\begin{figure}[H]
\begin{center}
\scalebox{0.52}{\includegraphics{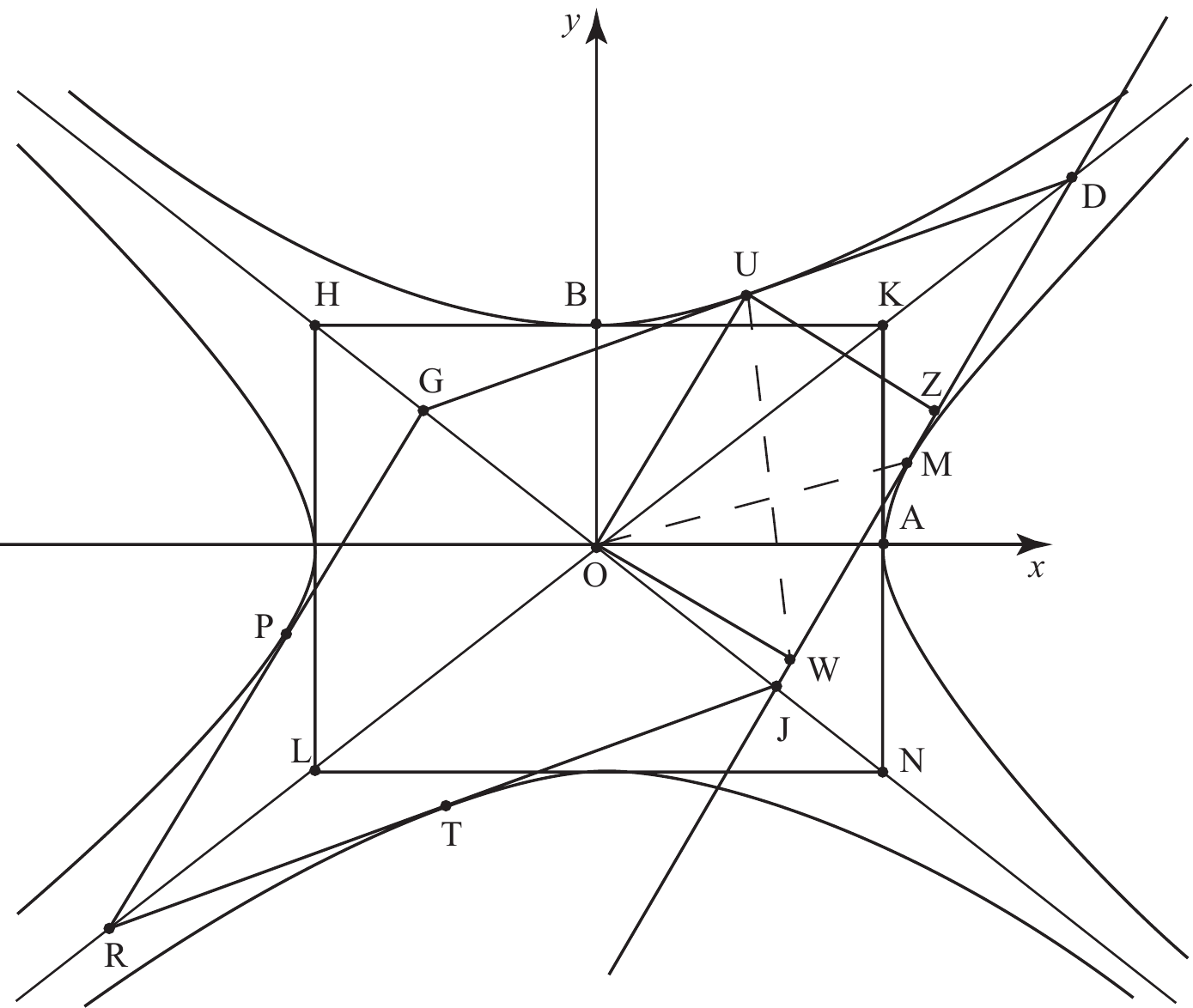}}
\end{center}
\caption{The geometric elements $(a,\, b,\, r,\, p)$ of a couple of conjugate hyperbol\ae}\label{f02}
\end{figure}
Using the second Apollonious' theorem we have ${\rm area} (RGDJ)={\rm area} (LNKH)=4ab,$ where $a$ and $b$ are the hyperbola semiaxses. This implies
$$
\displaystyle{\frac{1}{4} {\rm area} (RGDJ) = \frac{1}{4} {\rm area} (OMDU)={\rm area} (OAKB)=ab}
$$
being M midpoint of JD and U midpoint of GD. Now observe that rectangle OWZU, where OW (the Maclaurin's $p$ segment) perpendicular to the tangent line to the hyperbola in Z and OU conjugate semidiameter to the hyperbola, is equivalent to parallelogram OMDU, since they have the common base OU and the same altitude OW. Thus we have
\[
{\rm area} (OWZU)={\rm area} (OMDU)=ab.
\]
This follows from the equality
\[
{\rm area} (OWZU)=OU\times p.
\]
\end{proof}

\subsection{The hyperbola's \textit{pedal} equation}
In a cartesian orthogonal reference frame let us consider in the first quadrant the branch AQE of the hyperbola \eqref{ipxya} of \textit{eccentricity} $e=({a^{2}+b^{2}}/{a^{2}})^{1/2}$ and its conjugate \eqref{ipyx}: such branches, see \figurename~\ref{f03}, cut the axes at the points $A(a, 0)$ and $B(0, b)$ respectively. Let S be the common centre of both hyperbol\ae\, where we put our origin. On the branch AQE we consider a point E marked by the radius $\overline{SE}=r$, and draw to such a branch at E the tangent straight line $\tau\tau$, crossed at P by the straight line $nn$ through S perpendicular to it; we put\footnote{Notice that $r$ and $p$ are called \textit{pedal coordinates} of the hyperbola with respect to S. The name is coming from the \textit{pedal curve} of a curve with respect to a fixed (pedal) point, namely the locus of the points where the successive perpendiculars through that point cross the successive tangents to the curve.}: $\overline{SP}=p$. It is assumed, for instance, $a>b$; then, as a consequence of the first Apollonius theorem we get:
\begin{equation}\label{obel}
\overline{SE}^{2}-\overline{SH}^{2}=a^{2}-b^{2}=2a\varepsilon,  
\end{equation}
where
\begin{equation}
\varepsilon=a\left(1-\frac{e^{2}}{2}\right) \label{lengeps}
\end{equation}
is a convenient length, and its \textit{eccentricity} is $e.$ Putting $\overline{SE}=r$ in \eqref{obel} we get:
$
\overline{SH}^{2}=r^{2}-2a\varepsilon. 
$
\begin{figure}[ht]
\begin{center}
\scalebox{0.6}{\includegraphics{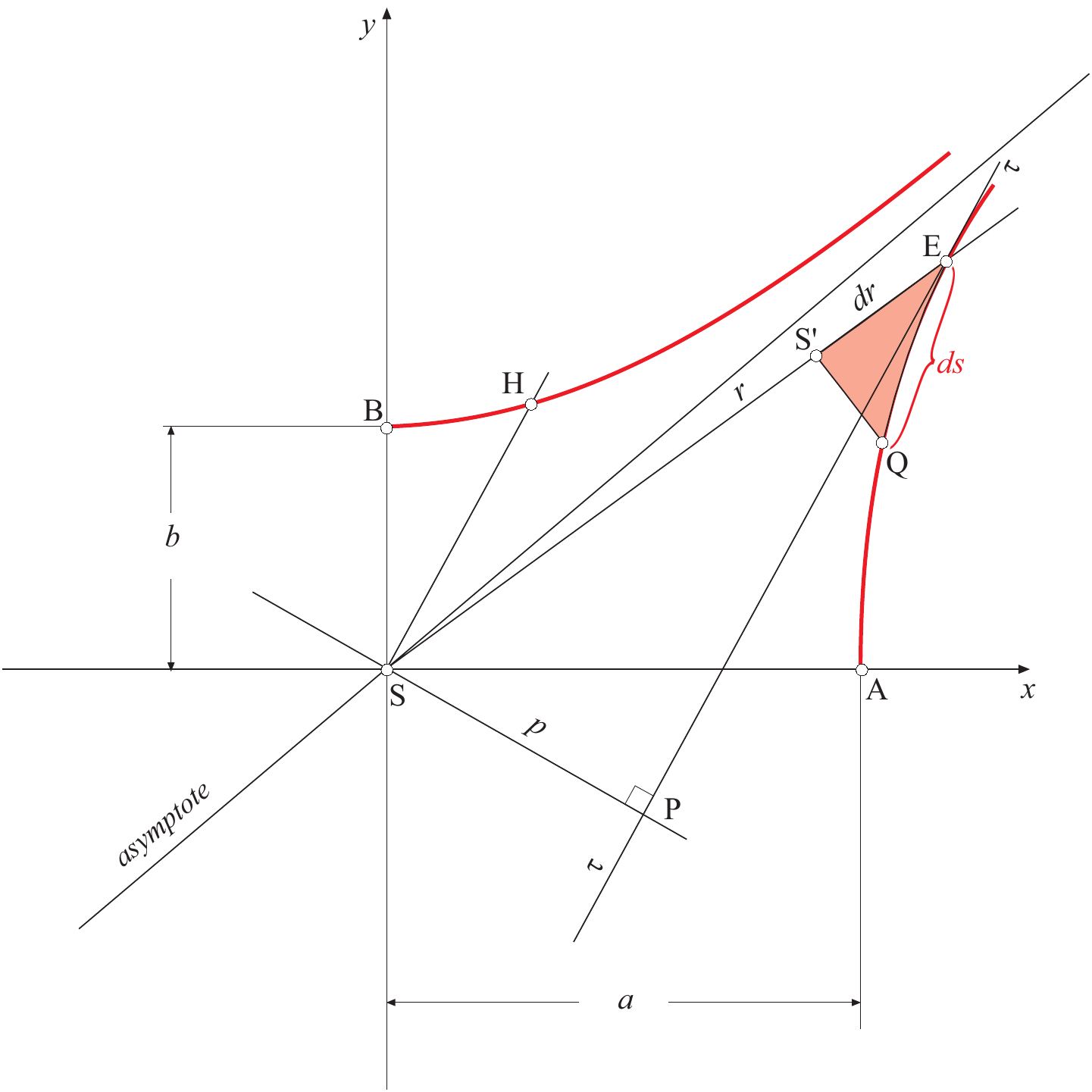}}
\end{center}
\caption{Pedal elements of a couple of conjugate hyperbol\ae}\label{f03}
\end{figure}
On the other side, the Apollonius corollary, formula \eqref{scorollary}
provides $\overline{SH}^{2}=a^{2}b^{2}/p^{2}.$  By comparison of $\overline{SH}^{2}$ expressions, Maclaurin gets:
\begin{equation}
r^{2}-2a\varepsilon=aX,\label{rho}
\end{equation}
where it has been defined\footnote{Maclaurin used $x$, but in order to avoid mistakes with the next sections, we prefer the capital letter $X$, whilst $x$ is kept for the abscissa.}:
\begin{equation}
X=\frac{ab^{2}}{p^{2}}.\label{ics}
\end{equation}
Solving to the radius $r$, taking its derivative with respect to $X$, one obtains:
\begin{equation}
{\rm d}r=\frac{a}{2\sqrt{2a\varepsilon+aX}}{\rm d}X. 
\end{equation}
In \figurename~\ref{f03}, however, considering the similar triangles $\triangle{\rm E}{\rm Q}{\rm S}$ and $\triangle{\rm E}{\rm P}{\rm S}$,
by equating the ratios of the hypotenuse to the greater cathetus, we have:
\begin{equation}
{\rm d}s=\frac{r}{\sqrt{r^{2}-p^{2}}}\,{\rm d}r.
\end{equation}
The infinitesimal arclength ${\rm d}s$ is then obtained with the hyperbola $(a,b)$ as a function of  $p$ and $r$. Plugging there the expressions \eqref{rho} and \eqref{ics} depending on  $X$ alone, Maclaurin gets:
\begin{equation*}
{\rm d}s=\frac{\sqrt{a}}{2}\frac{\sqrt{X}}{\sqrt{X^{2}+2\varepsilon X-b^{2}}}\,{\rm d}X
\end{equation*}
where  $X$ is variable with the point E and then with $p$; three fixed hyperbola parameters appear, i.e.: the semiaxes $a$ and $b$ and the arclength $\varepsilon$, whose two are independent.
At this point Maclaurin introduces the \textit{length of the tangent}, namely the segment bounded by $E$ and $P$:
$
\overline{EP}=\sqrt{r^{2}-p^{2}}.
$
Putting there the expressions of $r$ and $p$ as functions of $X$, we get $\overline{EP}$ as a function of $X$ alone, so that, taking the differential, one finds:
$$
{\rm d}(\overline{EP})=\frac{\sqrt{a}}{2x\sqrt{X}}\frac{X^2+b^2}{\sqrt{X^2+2\varepsilon X-b^2}}\,{\rm d}X.
$$
The \textit{differential excess} ${\rm d}\Delta$ concerning a single $E$-point of the hyperbola is the shift between the relevant tangent segment and the arclength whenever $X$ undergoes a change ${\rm d}X$, so that the radius changes of 
$$
{\rm d}r=\frac{\sqrt{a}}{2}\frac{1}{\sqrt{X+2\varepsilon}}\,{\rm d}X.
$$
Thus:
\begin{equation}
{\rm d}\Delta={\rm d}(\overline{EP})-{\rm d}s=\frac{\sqrt{a}}{2}\frac{1}{\sqrt{X^2+2\varepsilon X-b^2}}\,\frac{b^2}{X\sqrt{X}}\,{\rm d}X.\label{fond}
\end{equation}
 No doubt such a fluxion belongs to the third class of  Maclaurin's ranking of irrational ones.
\subsection{The excess \textit{p}-formula  and its consequences}

Starting from (\ref{fond}), minding the $X$-definition, Maclaurin puts\footnote{Really Maclaurin used $z$ but we changed to $\zeta $.}
$
p^{2}/{a}=b^{2}/{X}=\zeta, 
$
so that $dX=-b^{2}\,{\rm d}\zeta/\zeta^{2}$.  In such a way (\ref{fond}) becomes:
$$
{\rm d}\Delta=\frac{-\sqrt{a}}{2}\frac{\sqrt{\zeta }}{\sqrt{b^2+2\varepsilon \zeta -\zeta ^2}}\,{\rm d}\zeta.
$$
Such a formula\footnote{Notice that due to a print error in the 1742 edition, and not amended in the 1801 one, the factor \lq\lq2''  in the above formula at \textit{Fluxions}, page 245, row 6, is missing.} allows to construct the \textit{excess p-formula}. For the purpose, plugging in $\zeta $ as a function of $p$, after a little algebra, we have:
 \begin{equation}
{\rm d}\Delta=\frac{-p^2}{\sqrt{a^2 b^2+2\varepsilon ap^2-p^4}}\,\,{\rm d}p\label{dDp}
\end{equation}
namely the \textit{excess p-formula}, see \textit{Fluxions}, page 245, row 8. Let us go to the implications of \eqref{dDp}. First, let us put $a=b$, then  $\varepsilon=0$, equilateral hyperbola, so that:
 \begin{equation}
{\rm d}\Delta=\frac{-p^2}{\sqrt{a^4-p^4}}{\rm d}p\label{ddDp}.
\end{equation}
This connection really explains why Maclaurin, D'Alembert and Landen expended considerable effort over a curve, like the hyperbola, only marginally involved in the astronomic or ballistic computations. 
 In other words, given the hyperbola in \textit{pedal} form
 $$
 r(p)=a\left(2\varepsilon+a\frac{b^{2}}{p^{2}}\right),
 $$
 let us construct the excess ${\rm d}\Delta$ between the length of the tangent at a point and the arc of the hyperbola from its vertex to that $E$-point.
 Then the elastica curve can stem also by integrating the excess (\ref{ddDp}); or, as seen,  by rectifying the lemniscate.\footnote{ The shortest way driving to the rectification of the
lemniscate of equation $\rho^{2}=R^{2}\cos (2 \theta) $ is that assuming the polar anomaly as a variable: the elementary arc is
\[
{\rm d}s=R\frac{{\rm d}\theta}{\sqrt{\cos(2\theta)}}, 
\]
so that the one-quarter arclength is given by: 
\[
s_{\frac{1}{4}}=R\int_{0}^{\frac{\pi }{4}}\frac{{\rm d}\theta }{\sqrt{1-2\sin
^{2}\theta }}=\frac{R}{\sqrt{2}}\,\boldsymbol{K}\left( 
\frac{1}{\sqrt{2}}\right). 
\]
where $\boldsymbol{K}\left(\tfrac{1}{\sqrt{2}}\right)$ is the complete elliptic integral of first kind of modulus $k=\tfrac{1}{\sqrt{2}}$.} A further consequence of the $p$-formula \eqref{ddDp} of the differential excess consists of providing the finite excess, say $\Delta_\infty$, when the $E$-point, moving on the hyperbola, slides to infinity.\footnote{Even though the limit concept was a matter the of the XIX century, the notation used by Maclaurin was quite clumsy: the symbol \textit{lim} made its first appearance in the memoir \textit{Exposition elementaire des principes des calculs superieurs} by S. Lhuilier, Berlin, 1786.}
In such a case, both the tangential length and the arc, are really increasing without limits, so that their difference would appear indeterminate. We will show in the next section the relevant Maclaurin evaluation of the excess by means of a series expansion. 

When the E-point on the hyperbola slides from E to infinity, the perpendicular segment $p$ changes from the value $a$ to 0.
Minding the $e$ and $\varepsilon$ definitions, then \eqref{dDp} becomes:
$$
{\rm d}\Delta=\frac{p^{2}}{ab}\left(1+\frac{p^{2}}{b^{2}}\right)^{-1/2}\left(-\frac{a}{\sqrt{a^{2}-p^{2}}}\,{\rm d}p\right).
$$
Expanding the second factor in binomial series, Maclaurin obtains:
$$
{\rm d}\Delta=\frac{1}{ab}\left(-\frac{a}{\sqrt{a^{2}-p^{2}}}\,{\rm d}p\right)\left(p^{2}-\frac{p^{4}}{2b^{2}}+\frac{3}{8}\frac{p^{6}}{b^{4}}-\&c\right).
$$
After this -we think- he would have evaluated (putting $p=a\sin u$ in order to obtain three integrals of even powers of $\sin u$):
\[
\int_0^a\frac{p^2}{\sqrt{a^{2}-p^{2}}}\,{\rm d}p=\frac{\pi  a^2}{4},\,\int_0^a\frac{p^4}{\sqrt{a^{2}-p^{2}}}\,{\rm d}p=\frac{3 \pi  a^4}{16},\,\int_0^a\frac{p^6}{\sqrt{a^{2}-p^{2}}}\,{\rm d}p=\frac{5 \pi  a^6}{32}
\]
promptly leading to the final Maclaurin formula,\footnote{Maclaurin strangely writes $N$ instead of  $\pi/2$. The first to use $\pi$ definitely for the ratio of circumference to radius was William Jones (1675-1749) in his \textit{Synopsis Palmariorum Matheseos}, 1706. Euler adopted the symbol in 1737 and since then it became of general use.} which we can read at  \textit{Fluxions}, page 245, row 8:
\begin{equation}\label{maccessoappr}
\Delta_\infty=\frac{\pi a^{2}}{2b}\left(\frac{1}{2}-\frac{3a^{2}}{16b^{2}}+\frac{15a^{4}}{128b^{4}}-\&c\right).
\end{equation}
\subsection{What Maclaurin could not know\ldots}
Starting from \eqref{dDp}, let us consider the \textit{excess elliptic integral}
\[
\Delta_{\infty}=\int_0^a\frac{p^2}{\sqrt{a^2b^2+2\varepsilon ap^2-p^4}}\,{\rm d}p
\]
but recalling relationships between $\varepsilon,\,e,\,a,$ and $b$ and changing variable by $p^2=q$ we get:
\[
\Delta_{\infty}=\frac12\int_0^{a^{2}}\sqrt{\frac{q}{(b^2+q)(a^2-q)}}\,{\rm d}q.
\]
Let us refer to \cite{Grr}, page 263 integral 3.141-10: then the excess is given by the difference between two complete elliptic integrals of second and first kind:
\begin{equation}\label{maccesso}
\Delta_{\infty}=\sqrt{a^2+b^2}\,\boldsymbol{E}(k)-\frac{b^2}{\sqrt{a^2+b^2}}\,\boldsymbol{K}(k)
\end{equation}
where the elliptic modulus $k$ is given by
$
k=a/\sqrt{a^2+b^2}.
$
Now by the hypergeometric series expansions for $\boldsymbol{K}(k)$ and $\boldsymbol{E}(k)$
\[
\begin{split}
\boldsymbol{K}(k)&=\frac{\pi}{2}\,_{2}\mathrm{F}_{1}\left( \left. 
\begin{array}{c}
\frac12;\frac12 \\[2mm]
1
\end{array}
\right| k^2\right)=\frac\pi2\,\sum_{n=0}^\infty\left[\frac{(2n)!}{2^{2n}(n!)^2}\right]^2k^{2n}\\
\boldsymbol{E}(k)&=\frac{\pi}{2}\,_{2}\mathrm{F}_{1}\left( \left. 
\begin{array}{c}
-\frac12;\frac12 \\[2mm]
1
\end{array}
\right| k^2\right)=\frac\pi2\,\sum_{n=0}^\infty\left[\frac{(2n)!}{2^{2n}(n!)^2}\right]^2\frac{k^{2n}}{1-2n}
\end{split}
\]
expanding around $a=0$ relation \eqref{maccesso}, we find exactly \eqref{maccessoappr}.

Somewhat like to \eqref{maccesso} was also found by John Landen in 1780 as we will see in the next section.

\section{Landen}

\subsection{The hyperbola's theorem publishing history}
In the introduction to his \textit{Th\'{e}orie des fonctions analytiques}, 1797,
Joseph-Louis Lagrange (1736-1813) refers to John Landen (1711-1790) as a \textit{habile G\'{e}om\'{e}tre anglais}. Landen, famous among the most anti-academic mathematicians, had 
one final
purpose: improving Maclaurin and D'Alembert's results \cite{lan1}:
\begin{quote}
Mr. Maclaurin, in his Treatise of fluxions, has given sundry very elegant
theorems for computing the fluents of certain fluxions by means of elliptic
and hyperbolic arcs; and Mr. D'Alembert, in the Memoirs of Berlin Academy,
has made some improvement upon what had been written on that subject. But
some of the theorems given by those gentlemen being in part expressed by the
difference between an arc of an hyperbola and its tangent, and such
difference being not directly attainable,\ldots
\end{quote}
The hyperbola papers are in substance only one in its three variants 1771, 1775, 1780. Each of them holds the hyperbola rectification theorem, some corollary about the excess, and an application to a circular pendulum.


The first paper \cite{lan1} where Landen faced with the problem of hyperbola rectification was read on June 6, 1771 and enclosed in vol. 61 of the \textit{Philosophical Transactions of the Royal Society} under the title \textit{A disquisition\ldots} Its startup has been referred: his main criticism again of the limiting value of the excess will be analyzed later. There he refers explicitly to integration methods by means of arcs of hyperbola and ellipse and notes to have performed the hyperbola rectification, the proof of which he would show later in his second memoir. Afterwards Landen tried to apply this to the fluxion of time of a heavy bead freely descending from
rest along a circular arc.

The article of 1775 summarizes the conclusions of that issued four years earlier and describes his finding on page 285 with these words: 
\begin{quote}
Thus, beyond
my expectation, I find that the hyperbola may in general be rectified by means of two ellipses!
\end{quote}
After having applied again the fluxions to the time of a pendulum, Landen recalls that 
Maclaurin, Jakob
and Johann Bernoulli, and Leibniz deemed that the elastic curve could not be
constructed by the quadrature or rectification of conic sections. However, in \cite{lan2}, page 288:
\begin{quote}
the contents of this paper, properly applied, will evince that the
elastic curve (with many others) may be constructed by the rectification of
the ellipsis only, without failure in any point.
\end{quote}
Was Euler aware of know Landen's 1775 paper? Legendre's answer (\cite{Leg1}, p.
89) is negative:
\begin{quote}
Euler n'ait rien \'{e}crit \`{a} l'occason du M\'{e}moire de Landen,
imprim\'{e} dans les Transactions philosophiques de 1775, d' o\`{u} il faut
conclure que ce M\'{e}moire n'est pas parvenu \`{a} sa conaissance; car dans
l'hypoth\'{e}se contraire, cet illustre G\'{e}om\`{e}tre aurai sans doute,
suivant son usage, publi\'{e} sers propres r\'{e}flexions sur une
d\'{e}couverte analytique qui devait particuli\`{e}rement l'int\'{e}resser.%
\end{quote}
In any case, we do not agree with the level of importance commonly given to the 1775 article: Euler lived till to 1783 and he could have read its third version, published by Landen in \textit{Mathematical Memoirs}, 1780.

The final article \cite{lan3}
under the simple title \textit{Memoir of the ellipsis and the hyperbola}, was enclosed as II Memoir in the first volume of \textit{Mathematical memoirs}. The author, perhaps aware of the obscurity of his two first editions, looked for better clarity, that was not always achieved. In any case, we shall follow this closely.
Landen's aim was to investigate the
difference between an arc of an hyperbola and its tangent, since such differences were not directly attainable.

\subsection{Landen's excess formula as a consequence of  Maclaurin's}
In order to prove his hyperbolic theorem, Landen built as a first step, a formula for the hyperbolic excess which, either in his premises (Apollonius theorems, hyperbolic radius function of the $p$ normal length), or in the process (similar triangles, fluxions and chain of variables), tracks on the Maclaurin's scheme, as quoted by himself. 

Given the hyperbola of center S, semiaxes $a,\,b$ of equation \eqref{ipxya}, the hyperbola asymptotes are $y=\pm\,(b/a)\,x$. In a cartesian orthogonal reference frame we consider in the first quadrant the branch of the above hyperbola \eqref{ipxya} of \textit{eccentricity}
$
e=\sqrt{(a^{2}+b^{2})/a^{2}}.
$ 
Recalling $\varepsilon$ as introduced in formula \eqref{lengeps} we also recall $X$ from equation \eqref{ics}.
Taking the derivative with respect to $X$, he finds:
$$
{\rm d}(\overline{EP})=\frac{\sqrt{a}}{2x\sqrt{X}}\frac{X^2+b^2}{\sqrt{X^2+2\varepsilon X-b^2}}\,{\rm d}X.
$$
The \textit{differential excess} ${\rm d}\Delta$ concerning a single E-point of the hyperbola is the shift between the relevant tangent segment and the arclength whenever $X$ undergoes a change ${\rm d}X$, is found to be as in equation \eqref{fond}:
 Landen's 1780 demonstration adds nothing to Maclaurin's, and we will refrain from expanding on it. We prefer to display how, starting from \eqref{fond}, taken from \textit{Fluxions}, one can quickly express ${\rm d}\Delta$ given as the square root of a ratio of two quadratic binomials of a certain variable.

Let us start from the $(x,y)$ coordinates equation \eqref{ipxya} which will not be the final hyperbola of the theorem. The length set by Maclaurin as $\varepsilon,$ \eqref{lengeps}, is named by Landen
$f=(a^{2}-b^{2})/(2a).$ 
The perpendicular length $p$ stated by Landen as: $ p^{2}=m z$ which, compared with \eqref{ics}, provides the relationship linking the state variable $X$ by Maclaurin, to the Landen's $z$:
\begin{align}
X&=\frac{b^{2}}{z}\label{icchese}\\
{\rm d}X&=-\frac{b^{2}}{z^{2}}{\rm d}z.\label{de-icchese}
\end{align}
Then, plugging \eqref{icchese} and \eqref{de-icchese} in \eqref{fond} one gets:
\begin {equation}
{\rm d}\Delta=-\frac{1}{2}\frac{\sqrt{a}\sqrt{z}}{\sqrt{b^{2}+2fz -z^{2}}} {\rm d}z,\label{delta}
\end{equation}
which  can be read at Landen's \textit{Memoirs}, page 25, row 3.
Up to this point he makes a change of coefficients putting:
\begin{equation}
a=m-n;\quad  b=2\sqrt{mn},\label{semi}
\end {equation}
and defines, instead of $z$ a new variable:
\begin {equation}
 t^{2}=(m-n)^{2}-p^{2},\label{ti}
\end{equation}
so that $mz=m^{2}-t^{2}.$ In such a way the excess (\ref{delta}) will become:
\begin {equation}
{\rm d}\Delta= \sqrt{\frac{(m-n)^{2}-t^{2}}{(m+n)^{2}-t^{2}}}\,\,{\rm d}t .\label{excess}
\end {equation}
Having assumed $t$ defined by (\ref{ti}) as a new independent variable in the hyperbola equation \ref{ipxya}, then (\ref{excess}) provides the excess of the hyperbola whose semiaxes are given, see (\ref{semi}), by the difference and by the double geometrical mean of the coefficients  $(m,n)$ entering \eqref{excess}. Of course in there $m-n$ is the transverse semiaxis of the hyperbola, so that, wishing to compute its finite excess, $t$ shall be spanned minding \eqref{ti}.
The first of  \figurename~\ref{f21} shows the geometrical link of length segments $t,\,m-n$ and $p$.
\begin{figure}[p]
\begin{center}
\scalebox{0.8}{\includegraphics[angle=180]{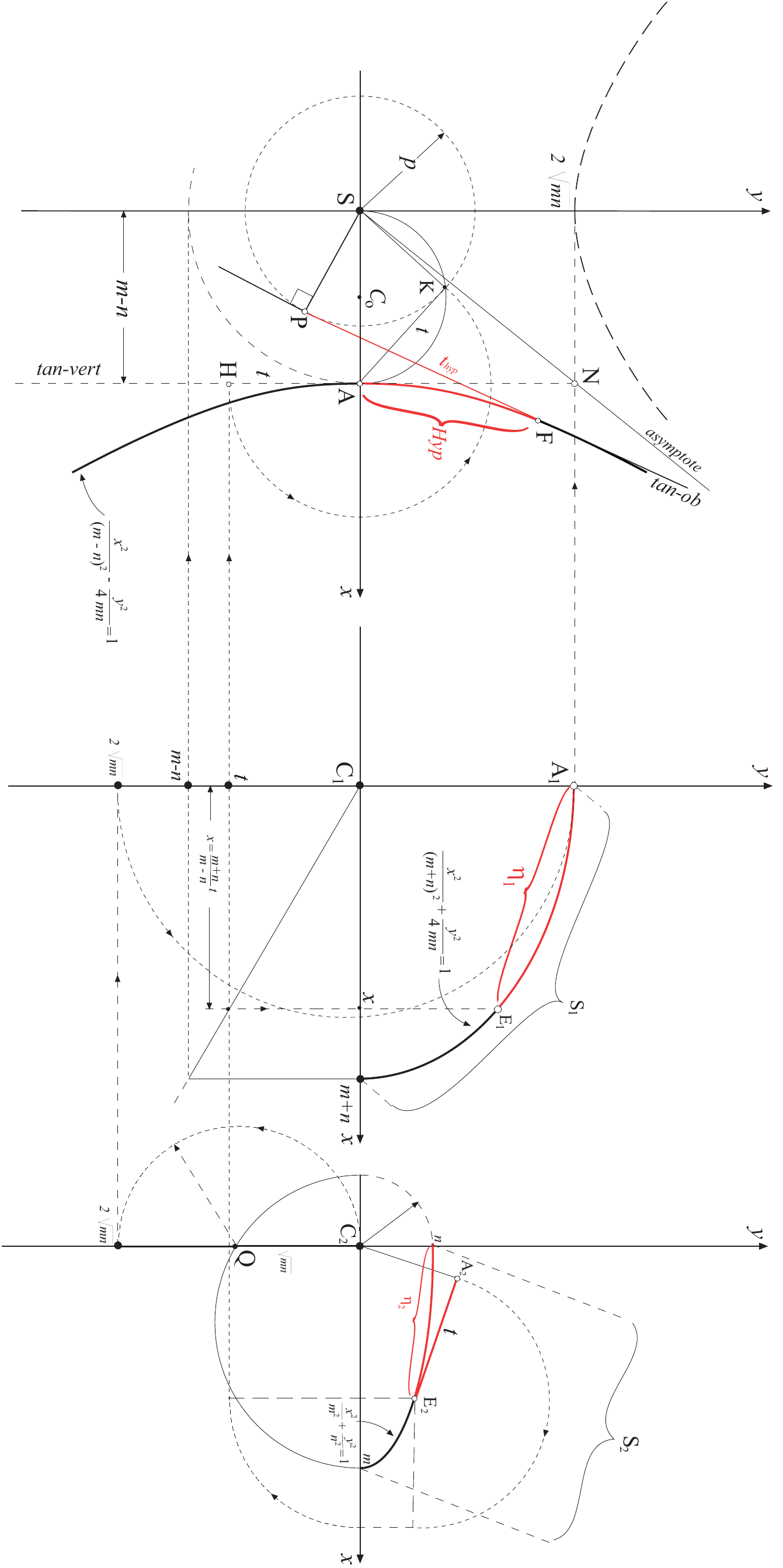}}
\end{center}
\caption{Our geometric view of all relationships in Landen theorem.}\label{f21}
\end{figure}

\subsection{Landen's auxiliary ellipses}
To deal with the last integral \eqref{excess} which was out of his capabilities, Landen followed the approach of integration by means of curves. He started from a first ellipse,
$$
\frac{x^{2}}{a'^{2}}+\frac{y^{2}}{b'^{2}}= 1
$$
whose differential arclength, by means of the characteristic triangle, can, after little work, be written as:
\begin {equation}
{\rm d}\eta_{1}=\sqrt{\frac{a'^{2}-g'x^{2}}{a'^{2}-x^{2}}}\, {\rm d}x\label {di-etauno}
\end {equation}
where 
$g'=(a'^{2}-b'^{2})/{b'^{2}}.$
Afterwards he specialized the curve assuming $a'=m+n,\, b'=2\sqrt{mn}$ where $m-n$ and $2\sqrt{mn}$ are the semiaxes of the hyperbola, and 
$
g'=\left[(m-n)/(m+n)\right]^{2}.
$
Making the change of variables $x\mapsto t$: 
\begin{equation}
x=\frac{m+n}{m-n}\,t, \label{xt}
\end{equation}
the finite $\eta_{1}$ first ellipse's  arclength will be the fluent\footnote{Of course what above means that the required elliptic arc has been taken between abscissa zero and $x$ and is given by the definite integral: $$\int_{0}^{\frac{m-n}{m+n}x}\sqrt{\frac{(m+n)^{2}-t^{2}}{(m-n)^{2}-t^{2}}}{\rm d}t=(m+n)\,E\left(\arcsin\frac{x}{m+n},\frac{m-n}{m+n}\right)$$
where we used entry 3.171-14 page 306 of \cite{Grr}.} of:
$
\sqrt{((m+n)^{2}-t^{2})/((m-n)^{2}-t^{2})}.
$
Let the second ellipse have semiaxes $m$ and $n$, so that 
$
g=(m^{2}-n^{2})/{m^{2}}.
$
We make reference to the mid  \figurename~\ref{f21}: chosen along the ellipse a point E of abscissa $x$, we draw the tangent at this point, whose segment $t$ is limited by E and the intersection P with the perpendicular sent from the $C_{1}$ ellipse's center.

By means of standard procedures, one can find the value of $\overline{EP}=t$ as a function of $x$:
\begin {equation}
t=gx\sqrt{\frac{m^{2}-n^{2}}{m^{2}-gx^{2}}} .\label {ti of x}
\end {equation}
Solving to $x^{2}$:
\begin {equation}
2gx^{2}=t^{2}+gm^{2}\mp\sqrt{(m^{2}-n^{2} )^{2}-2t^{2}(m^{2}+n^{2})+t^{4}}.\label {squarex}
\end {equation}
 Landen discards the minus sign, so that, after differentiation, writing the radicand as
$$
\left[(m-n)^{2}-t^{2}\right]\left[(m+n)^{2}-t^{2}\right]
$$
by \eqref{squarex} he gets the fluxion:
$$
2gx{\rm d}x=t{\rm d}t+\frac{(m^{2}+n^{2})t\,{\rm d}t -t^{3}\;{\rm d}t}{\left[(m-n)^{2}-t^{2}\right]^{1/2}\left[(m+n)^{2}-t^{2}\right]^{1/2}}
$$
as written in his \textit{Memoirs}, page 32, row 4 from bottom.
Dividing last formula\footnote{About the above formula, we read in \cite{Smadja} the comment:
\begin{quote}
La r\'eduction de l'arc hyperbolique \`{a} deux arcs elliptiques r\'esulte en effet de la d\'ecomposition alg\'ebrique d'une expression rationnelle en une somme de deux expressions rationnelles plus simples, en sorte que la proc\'edure alg\'ebrique op\`{e}re sur les termes fluxionnels en mettant de c\^{o}t\'e ce que Landen nomme leur \lq\lq figurative sense''.
\end{quote}
} to $t$ and minding \eqref{ti of x}, he gets:
$$
\frac{gx\,{\rm d}x}{t}=\sqrt{\frac{m^{2}-gx^{2}}{m^{2}-x^{2}}} {\rm d}x, \label{pippow}
$$
which has the same right hand side of (\ref{di-etauno}), which provides the elliptic elementary arc ${\rm d}\eta_{1}$  whose semiaxis $a$ has been changed to $m$:  such a new elementary arclength we are naming ${\rm d}\eta_{2}$.
Therefore:
\begin {equation}
{\rm d}\eta_{2}=\frac{gx{\rm d}x}{t}=\frac{1}{2}\, {\rm d}t+\frac{1}{4} \sqrt{\frac{(m-n)^{2}-t^{2}}{(m+n)^{2}-t^{2}}}\, {\rm d}t+\frac{1}{4} \sqrt{\frac{(m+n)^{2}-t^{2}}{(m-n)^{2}-t^{2}}}\,{\rm d}t ,\label{deta2}
\end {equation}
where the element $dt$ is tangential to the $(m,n)$ ellipse.
\subsection{The finite excess and Landen's geometric Theorem}
Before integrating formula \eqref{deta2} it will be observed that at its right hand side the second term means the elementary excess of hyperbola of semiaxes  $m-n$ and $2\sqrt{mn}$, whilst the third term is the elementary arc (\ref{di-etauno}) of the ellipse of semiaxes  $m+n$ and $2\sqrt{mn}$. In such a way, by integration we will get:
$$
\eta^{(A_{2}E_{2})}=\frac{1}{2}\overline{PE}+\frac{1}{4}(\overline{FP}-\overline{FA})+\frac{1}{4}\eta^{(A_{1}E_{1}).}
$$
Writing  $\overline{PF}= t_{Hyp}$ and recalling $\overline{PE}=t$, we get:
\begin {equation}
\overline{AF}=Hyp=t_{Hyp}+2t+\eta^{(A_{1}E_{1})}-4\eta^{(A_{2}E_{2})}.\label{llanden}
\end {equation}
We now highlight to how the \textit {Landen  hyperbolic theorem itself  can be read as belonging to the family of integrable combinations\footnote{We will come back later about them.}}: it states in fact that even if each of the three arcs of hyperbola + ellipse 1+ ellipse 2 is not algebraically integrable, nevertheless their linear combination
$
Hyp-\eta^{(A_{1}E_{1})}+4\eta^{(A_{2}E_{2})}
$
is algebraically integrable and equates the addition of segments $t_{Hyp}+2t$.

The relationship \eqref{llanden} displays the Landen theorem:

\begin{teorema}
The arc $\overline{AF}$ of hyperbola from a vertex A to F, whose elementary excess is given by \eqref {excess} and then having semiaxes $(m-n)$ and $2\sqrt{mn}$, is computable by \eqref{llanden}, where: 
\begin{enumerate}
\item[a)] $\overline{FP}$ is the straight segment of pedal \footnote{Such a tangent is said to be \lq\lq pedal'' because it is bound by the P-foot of the perpendicular drawn through a fixed "pedal point" to it. The \lq\lq pedal curve'' of a given line is the locus of all the P-points when the tangent varies continuously along the profile of the line. The pedal curve of a rectangular hyperbola with the pedal point at its focus is a circle, but with pedal point at its center, one will find a lemniscate of Bernoulli.} tangent to hyperbola  at F;

\item[b)] $t_{Hyp}$ is the straight segment  of pedal tangent to the ellipse of semiaxes $(m,n)$;

\item[c)] $\eta^{(A_{1}E_{1})}$ is the arc of the ellipse 1 of  semiaxes $(m+n, 2\sqrt{mn})$;

\item[d)] $\eta^{(A_{2}E_{2})}$ is the arc of the mentioned ellipse 2

\end{enumerate}

\end{teorema}
 According to G. N. Watson, who has really understood \cite{Marquis} all the above machinery of ellipses:
\begin{quote}
Pairs of ellipses whose semiaxes are related in the manner of two ellipses of this problem are said to be connected by Landen's transformation. In the hands of Legendre the transformation became a most powerful method for computing elliptic integrals.
\end{quote}
Let us add our 12-step operations' sequence in order to achieve a complete geometrical overview of all the machinery \eqref{llanden}, starting from the sole knowledge of $m$ and $n$. We refer to  \figurename~\ref{f21}.
\begin{enumerate}
\item[1)] By $m>0$ and $n>0$, we construct the segments $m-n$ and $2\sqrt{mn}$: they will be semiaxes of the hyperbola which we draw pointwise on a $(S, x, y)$ frame, branch $(x>0, y>0)$, vertex A. Let us draw the asymptote \textit {asy} too.

\item[2)] We draw the vertical straight line \textit {vert-tan} passing through A.

\item[3)] Half-circle over the diameter $\overline{AS}$ and whose centre is $C_{0}$.

\item[4)] Centre $C_{1}$, we draw the quadrant  $(x> 0, y> 0)$ of the ellipse 1:
$$
\frac{x^{2}}{(m+n)^{2}}+\frac{y^{2}}{4mn}= 1
$$
of semiaxes  $m+n$ and $2\sqrt{mn}$, the last quantity being provided as ordinate of the intersection N between \textit {vert-tan} and \textit {asy}. 

\item[5)] Let us draw $(x>0, t>0)$ the straight line of \eqref{xt} connecting the origin to Z$\equiv(t=m-n, x=m+n)$ being the $t$ reference axis heading downwards.

\item[6)] Centre $C_{2}$, we draw the quadrant  $(x>0, y>0)$ of the ellipse:
$$
\frac{x^{2}}{m^{2}}+\frac{y^{2}}{n^{2}}= 1
$$
of semiaxes  $m$ and $n$.
 
\item[7)] Choosing a point E on the second ellipse, and the relevant E-tangent, we construct the perpendicular from $C_{2}$ to it, P is found:  let it be $\overline{A_{2}E_{2}}=t$ which we transfer as $\overline{C_{1}t}$, $\overline{AH}$, and so on.

\item[8)] From the third diagram we enter the second one and, given $t$, by means of the projecting horizontal line, we get both $x$ and $\eta^{(A_{1}E_{1})}$.

\item[9)] Entering the hyperbola, let H be the intersection of the mentioned projecting horizontal line with \textit {vert-tan}. 

\item[10)] Centre in A, spread $\overline{AH}$, we get a circle which crosses that of centre $C_{0}$, finding K.

\item[11)] The triangle AKS is right in K, so that, if $\overline{AK}=t$, $\overline{AS}=m-n$,  so that:
$
\overline{SK}^{2}=\overline{AS}^{2}-\overline{AK}^{2}
$
then by \eqref{xt}, it shall be: $\overline{SK}=p$.

\item[12)] Known $p$,  we draw the circle of centre S and radius $p$, and consider the straight line 
\textit {tan-ob} touching \textit{simultaneously} it and the branch  $(x> 0, y> 0)$ of the hyperbola. Such a tangent is unique.\footnote{Apart from each possible intuition, the thing can easily be set  analytically. First, one assumes a named $(x_{P},y_{P})$ point of the circle and writes its tangent there, which, by means of the circle equation, will be parametrized on the abscissa of such a point. The same for the hyperbola on the abscissa of its generic point $(x_{F},y_{F})$. Imposing that both slopes and intercept of two straight lines shall be the same, one gets a 2-equation, 8th degree non-linear algebraic system holding the required abscissas. Then we find four straight lines, but restricting to the hyperbola's upper half branch $(x>0,y>0)$, there will only be one bitangent, say the straight line named \textit{tan-ob} in \figurename~\ref{f21}.} In this way we find both points P and F. 
\end{enumerate}

We have then shown the way of drawing the finite hyperbolic excess  between the points A and F of the $(m, n)$ hyperbola parametrized on the $t$ value of the tangent length cut on the second $(m,n)$ ellipse. Furthermore we displayed how the fixing of $t$ is equivalent to the fixing of the upper F point on the hyperbola.

\subsection{Simpson's Fluxions and his treatment of hyperbolic excess}

Thomas Simpson's (1710-1761) \textit{A New treatise of Fluxions} was issued in 1737 (five years before Maclaurin's); in 1750 he published a new edition of it, which, however, he had wished to be considered as new work rather than a new edition of an old one. We consulted the posthumous reprint, 1776, where, page 509 of volume I, Simpson describes the problem:
 \begin{quote}
To determine the difference between the length of the arch of a semi-hyperbola infinitely produced, and its asymptote.
 \end{quote}
 Simpson could not have seen any of Landen papers (1771, 1775, 1780) so this is either his personal elaboration (following Maclaurin whose treatise had went out in 1742), or a reflection of some discussion on the subject he had with Landen.
 Let us give some elements on his treatment. Both Maclaurin and Landen in treating such a subject make use of non-cartesian variables and for their purpose produce pedal coordinates, polar anomalies, lengths of the tangent, and so on, integration variables not all having a direct geometric visibility. On the contrary, Simpson makes use only of the cartesian orthogonal coordinates $(x,y)$. If $a$ and $b$ denote the hyperbola semiaxes, being the first over the $x$ axis, of course, $y=(b/a) \sqrt{x^{2}-a^{2}}$, for its elementary arclength he obtains:
\[
 {\rm d}s=-\frac{a}{\delta}\, \frac{\sqrt{1-\delta^{2}u^{2}}}{u^{2}\sqrt{1-u^{2}}}\,{\rm d}u
 \]
where $ \delta^{2}=a^{2}/(a^{2}+b^{2}).$
Expanding the power $1/2$ by means of the binomial theorem and integrating term by term, the first integrated one gives by itself the length of the tangent. In such a way: 
\begin{quote}
therefore the difference between the arc and the asymptote will be equal to the fluent of the remaining terms in the difference sought.
\end{quote}

Finally, the hyperbola arclength is computed in some particular cases of interest.

\subsection{Fagnano's theorem on elliptic arcs whose difference is rectifiable}
We have been ventured to mention the \textit{integrable combinations} which Johann Bernoulli had introduced, 1695. He showed that on some curves two arcs whose difference was rectifiable could be found, although each separate arc could not be rectified.
 
 That subject is quite close to the work of  the Italian mathematician Giulio Carlo Toschi di Fagnano (1682-1766), Fellow of the Royal Society since 1723, who published in 1716 a paper\footnote{In 1750 this article and almost all his writings entered his collected works entitled \textit{Produzioni matematiche}. Finally in 1911-12 a modern complete edition \cite{Gambioli} was issued, with better figures, intelligible formulas, letters, and a detailed biography of Fagnano.} \cite{Giorn}, which has some connection with that of Landen's.
 Fagnano proved that on some curves it is possible to find infinite arcs whose differences can be algebraically found, \textit{even if the single arcs cannot be rectified}. Or, analytically speaking, infinite differential combinations integrable over those curves.
We can read in \cite{Leg1}:
\begin{quote}
Un G\'{e}om\`{e}tre italien dÕune grande sagacit\'{e}, ouvrit la route \`{a} des sp\'{e}culations plus profondes. Il prouva que sur toute ellipse ou sur toute hyperbole donn\'{e}e, on peut assigner, dÕune infinit\'{e}
 de mani\`{e}res, deux arcs dont la diff\'{e}rence soit \'{e}gale \`{a} une quantit\'{e}
 alg\'{e}brique. Il d\'{e}montra en m\`{e}me temps que la courbe nomm\'{e}e lemniscate jouit de cette singuli\`{e}re propri\'{e}t\'{e}, que ses arcs peuvent \`{e}tre multipli\'{e}s ou divis\'{e}s alg\'{e}briquement, comme les arcs de cercle, quoique chacun dÕeux soit une transcendante dÕun ordre sup\'{e}rieur.
\end{quote}
Fagnano was especially successful with the cubic parabola, the lemniscate\footnote{The general lemniscate is known as a \textit{Cassini oval}, 1680. A  special kind of it, called \textit{hyperbolic}, was considered by Jakob Bernoulli, and investigated by Fagnano since 1750, and by Euler, 1751,1752.} and the ellipse, but he is now better remembered in connection with the latter.
On a quadrant AQPB (see \figurename~\ref{f22} ) of a given C-centered ellipse, Fagnano found pairs of points, like P and Q, such that the difference of the arcs  $\overline{BP}$ and  $\overline{AQ}$ is rectifiable by ordinary integration. It was afterwards found that if perpendiculars CM and CM' are drawn from C, onto the tangents at P and Q, then $\overline{PM}=\overline{QM'}$; and that each of these is equal to the difference between the two arcs mentioned.
\begin{figure}[H]
\begin{center}
\scalebox{0.75}{\includegraphics{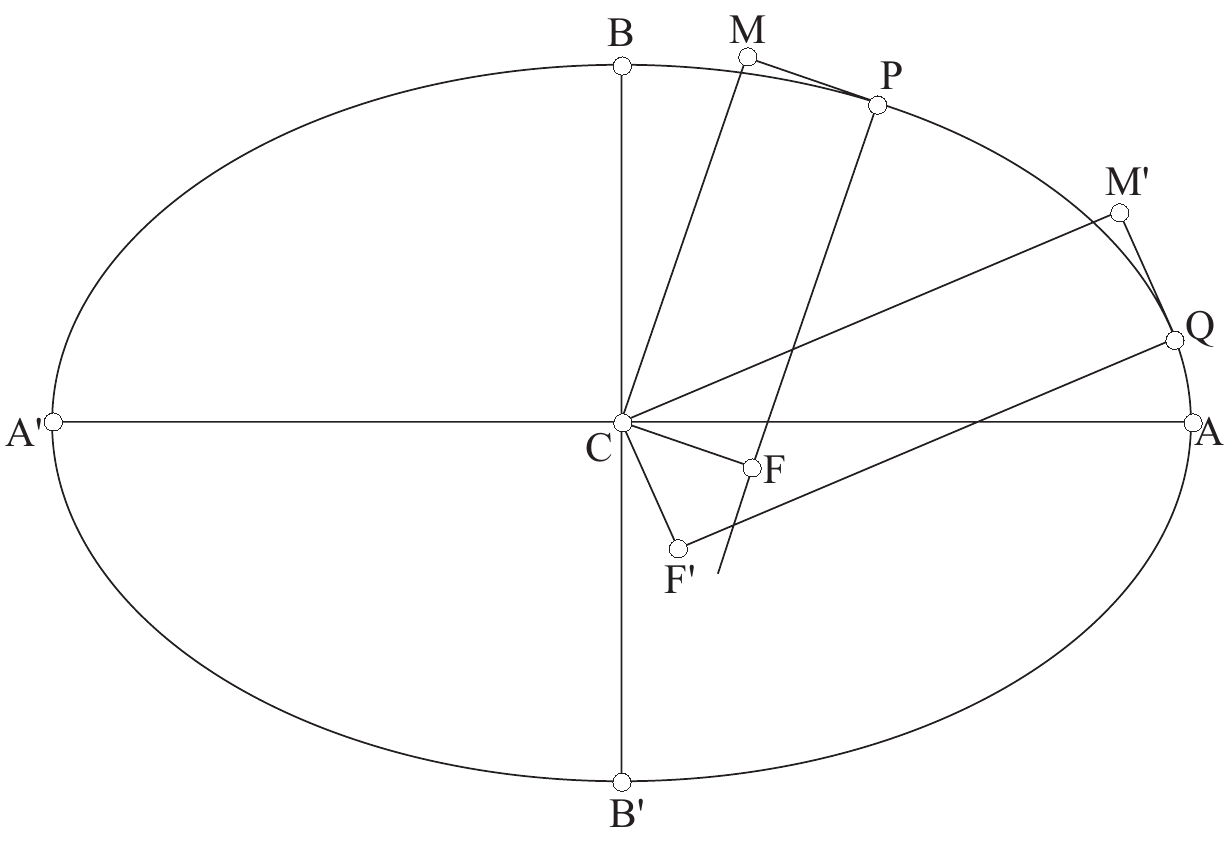}}
\end{center}
\caption{Fagnano theorem on ellipse's arcs}\label{f22}
\end{figure}

Fagnano's work, though with methods different from his  \textit{Produzioni matematiche}, is considered quite extensively in \cite{Green}, pages 182-189.

The Italian mathematician Francesco Siacci (1839-1907) see \cite{Siacci}, collecting some Fagnano's  theorems and further contributions (alternative proofs, additions, details, corollaries, geometrical constructions, and so on) due to:  Leonard Euler (1707-1783), Adrien Marie Legendre (1752-1833), Augustin Luis Cauchy (1789-1855), C. K\"{u}pper\footnote{About K\"{u}pper we know that he was author of the item \cite{kup}, but we did not succeed in finding anything else about this German professor of the 19th century.}, William Wallace (1768-1843), John Brinkley (1763-1835), Pierre Verhulst (1804-1849), Paul Serret (1827-1898), Michel Chasles (1793-1880), and  Johann August Gr\"{u}nert (1797-1872), \textit{surprisingly did not ever cite Landen!}
On the contrary, in his memoir \textit{Of the ellipsis and hyperbola} about the hyperbolic excess theory, in articles 4, 9 and 13, see \figurename~\ref{f23},
\begin{figure}[H]
\begin{center}
\scalebox{0.99}{\includegraphics{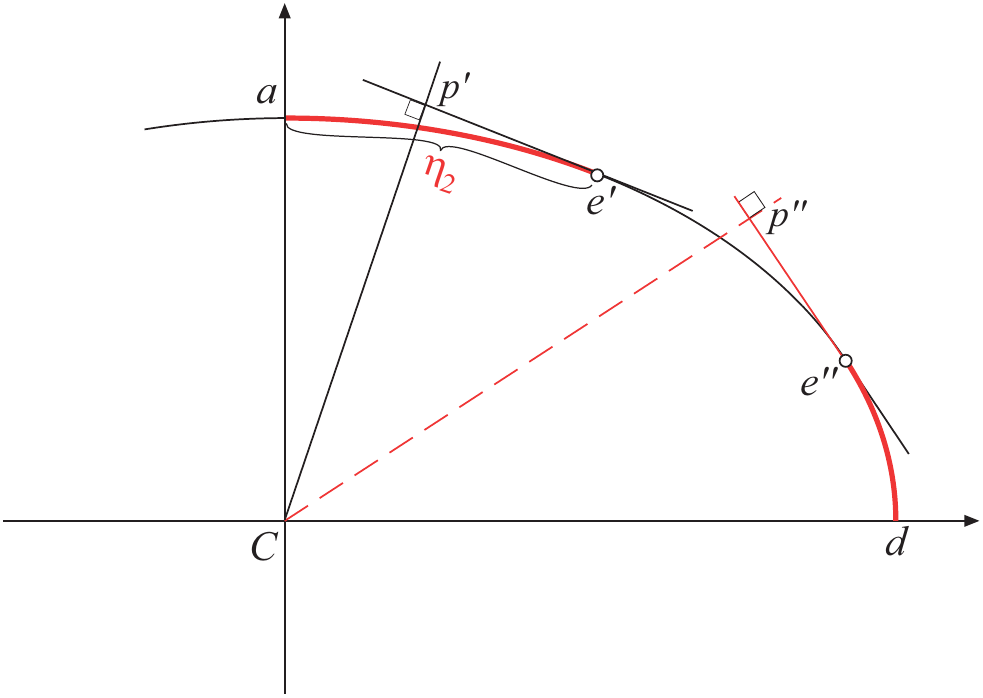}}
\end{center}
\caption{Scheme of Landen's proof that a couple of elliptic arcs has a difference given by the segment $\overline{p'e'}$}\label{f23}
\end{figure}
\noindent by purely fluxional means, Landen constructs a theorem (by himself not specially highlighted but) whose statement is:
\begin{quote}
If  $\overline{p'e'}$ and  $\overline{p''e''}$ be equal tangents to the ellipsis $ae'e''d$; the arc $\overline{ae'}$ will be equal to the arc $\overline{de''}$+the tangent $\overline{p''e''}$
\end{quote}
\textit{namely nothing but the ellipse of Fagnano's theorem}. 
The theorem of \lq\lq Comte Fagnani'' was cited and explained by\footnote{Siacci's confusion is due to the Legendre relevant papers, namely \cite{Leg4} and \cite{Leg3}, inserted in succession in the same volume of the \textit{M\'{e}moires de l'Acad\'{e}mie Royale.}} \cite{Leg4}, and again in \cite{Leg2}, page 44.
On the contrary, we ignore whether Landen in 1780 was aware of Fagnano's article editions (1716; 1750) or not; in any case he preferred to discard it, showing his own proof. In any case he did not ever cite Fagnano, and we are inclined to believe he did not know about Fagnano's paper at all.
A comparison between the proofs of Fagnano and Landen will not be performed here.
\subsection{Landen and the hyperbola limit excess}
\subsubsection{His pendulum-based motivation}
In the paper of 1780, articles $2\div14$ concerning the hyperbola, are followed by articles 15 and 16 about the ellipse.
At page 34 Landen meets a differential of the type:
\begin{equation}
\frac{{\rm d}z}{\sqrt{z}\sqrt{n^{2}+2fz-z^{2}}},\label{ahhh}
\end{equation}
which he solves through arcs of ellipse, namely doing an integration by means of curves. Even if the second memoir holds application to a pendulum, the third memoir entitled \textit{Of the descent of a body in a circular arc}, is completely devoted to the heavy body's motion \textit{in vacuo} along a circular arc.
At page 37 the fluxion of the time of descent is set as:
$$
\frac{1}{2}\frac{r h^{-1/2}x^{-1/2}\dot{x}}{\sqrt{2b r-b^{2}-2 \overline {r-b}  x-x^{2}}},
$$
$h,\,r,\,b$ being  some constants and $x$ a variable distance during the bead motion.\footnote{Landen's above fluxion in modern terms reads as
$$
\frac{1}{2}\frac{r h^{-1/2}x^{-1/2}}{\sqrt{2b r-b^{2}-2({r-b})x-x^{2}}}\,{\rm d}x.
$$
}
The above integral is of the same type described at \eqref{ahhh}, which is exactly the same of \eqref{delta} previously met as the differential hyperbolic excess of Maclaurin's \eqref{fond}.
What the above confirms the practical tendency kept by Landen, who, even when faced with theoretical questions, saw calculus as a tool to be improved on more and more, in order to solve recreational problems, rectification of curves, problems of dynamics and algebraic equations as well.

Whenever the F point along the hyperbola goes to infinity, the excess becomes $\infty-\infty$, and then indeterminate.
Landen then provides different approaches for obtaining the value of the limit hyperbolic excess he calls $L$.
\subsubsection{The limit excess: first proof, 1771, 1775.}

The calculation is in article \cite{lan2} of 1775, organized as it follows.
Solving \eqref{llanden} to the arc of the second ellipse, he gets:
$$
\eta^{(A_{2}E_{2})}=\frac{1}{2} t+\frac{1}{4}(t_{Hyp}-Hyp)+\frac{1}{4}\eta^{(A_{1}E_{1})}.
$$
But when the hyperbola's point tends to infinity, then, see \eqref{ti}, $p\to 0$ so that $t\to m-n$. For the first ellipse the abscissa of point E becomes, see \eqref{xt}, $m+n$ and then the arc fills the whole quadrant, say $S_{1}$, while that relevant to the second ellipse becomes:
$
1/2S_{2}+1/2(m-n)
$
$S_2$ being the second ellipse quadrantal arc. Landen claims to have demonstrated the above relationship at art. 10 of his paper, 1771, an extremely long and useless subject which will be omitted here. Inserting all this in \eqref{llanden}, he gets:
$$
\frac{1}{2}S_{2}+\frac{1}{2}(m-n)=\frac{1}{2}(m-n)+\frac{1}{4}L+\frac{1}{4}S_{1}
$$
or:
\begin {equation}
L=2S_{2}-S_{1},\label{luna}
\end {equation}
so that the indetermination is solved: \textit{the limit $L$ of the excess of a $(m,n)$ hyperbola is found as a simple combination of quadrantal arcs of the auxiliary ellipses 1 and 2 whose semiaxes closely depend on $m$ and $n$}.
Two short presentations that attempt to translate the machinery of Landen's hyperbola theorem and its auxiliary ellipses into mathematical language of the 1900s on a pure analytical basis, without reference to Maclaurin or the rectification landscape, are due to \cite{Marquis} and to \cite{Cantor4} pp. 842-847, both readable papers, but where the historical context is deliberately lost.
\subsubsection{The Landen  limit excess final proof, 1780.}
With reference to \figurename~\ref{f21}, the starting point is:
 \begin {equation}
\overline{AF}=Hyp=t_{Hyp}+2t+\eta_{1}-4\eta_{2}\label{anden}.
\end {equation}
Landen applies\footnote {There is some notation discrepancy because Landen changed it very frequently, now calling $Q$ the arc of the first ellipse, $R$ that of the second, while their quadrantal arcs are named $E$ and $E''$ respectively. For a better reading we tried to keep, as far as possible, the same symbols through Maclaurin and all of Landen's variants: furthermore we will use a lighter notation for the incomplete arcs of ellipse, writing   $\eta_{1}$ and  $\eta_{2}$ instead of  $\eta^{(A_{1}E_{1})}$ and $\eta^{(A_{2}E_{2})}$ respectively.}  to the second ellipse  what is merely Fagnano's theorem on the elliptic arcs whose difference is rectifiable:
\begin {equation}
\overline{p'e'}=\overline{p''e''}=t=\eta_{2}-\overline {de''}.\label{Fag}
\end {equation}
Plugging (\ref{Fag}) in (\ref{anden}), he gets:
$
 Hyp=t_{Hyp} +\eta_{1}-2\eta_{2}-2\overline{de''}.
$
Naming $S_{2} $ the quadrantal arc $ad$, by \figurename~\ref{f21} we have:
$
\eta_{2}=S_{2}-\overline{e'e''}-\overline{e''d}
$
so that (\ref{anden}) becomes:
$
Hyp=t_{Hyp}-2S_{2}+2\overline{e'e''}+\eta_{1},
$
so that the excess will be given by $2S_{2}-2\overline{e'e''}-\eta_{1}.$ When the point on the hyperbola is going to infinity, we know the variable $t$ attains its maximum value $m-n$, so that the arc of first ellipse fills all its quadrant, assuming the value $S_{1}$.
Such a maximum value is unique so that it is relevant to only one arc of the second ellipse. Otherwise speaking, two different arcs $ae'$ and $e''d$ cannot exist at which $t$ attains its maximum. Then $e'\equiv e''$so that $\overline{e'e''}\to 0.$ The conclusion is then \eqref{luna}, again.

\subsection{Our direct approach to the excess through elliptic integrals.}

\begin{teorema}\label{cess}
For the hyperbola of equation \eqref{ipxya} with $a\geq b>0$, the excess is: 
\begin{equation}\label{rcesso}
\mathscr{E}(0)=\sqrt{a^2+b^2}\, \boldsymbol{E}\left(\frac{a}{\sqrt{a^2+b^2}}\right)-\frac{b^2}{\sqrt{a^2+b^2}}\,\boldsymbol{K}\left(\frac{a}{\sqrt{a^2+b^2}}\right)
\end{equation}

\end{teorema}

\begin{proof} Rotate negatively by $\pi/2-\arctan(b/a)$ the axes so that the asymptote of equation $y=(b/a)x$ will coincided with the vertical axis. Then applying the transformation of coordinates
\[
\begin{cases}
x'=x\cos\left(\arctan\frac{b}{a}-\frac{\pi}{2}\right)-y\sin\left(\arctan\frac{b}{a}-\frac{\pi}{2}\right)\\
y'=x\sin\left(\arctan\frac{b}{a}-\frac{\pi}{2}\right)+y\cos\left(\arctan\frac{b}{a}-\frac{\pi}{2}\right)
\end{cases}
\iff
\begin{cases}
x'=\dfrac{b x + a y}{\sqrt{a^2+b^2}}\\[3mm]
y'=\dfrac{-a x + b y}{\sqrt{a^2+b^2}}
   \end{cases}
\]
changing \eqref{ipxya} in
\begin{equation}\label{ha}
y=\frac{\left(a^2-b^2\right)x^2 +a^2 b^2}{2 a b x}
\end{equation}
and where we will omit $^\prime$. The hyperbola \eqref{ha} vertex is:
\[
{\rm A}=\left(\frac{a b}{\sqrt{a^2+b^2}},\frac{a^2}{\sqrt{a^2+b^2}}\right)
\]
Taken  $\varepsilon>0$ close enough to zero, the straight line touching hyperbola \eqref{ha} at ${\rm E}=(\varepsilon,y(\varepsilon))$ is:
\[\tag{$\tau$}
y=\frac{a b}{\varepsilon }-\frac{x \left(a^2 b^2-\varepsilon ^2
   \left(a^2-b^2\right)\right)}{2 a b \varepsilon ^2}
\] 
so that the normal to $\tau$ going out from the origin has equation:
\[\tag{$\nu$}
y=\frac{2 a b \varepsilon ^2}{a^2 b^2-\varepsilon ^2 \left(a^2-b^2\right)}\,x
\]
and, if ${\rm P}:=\tau\cap\nu,$ then
\[
\overline{{\rm PE}}=\frac{a^4 b^4-\varepsilon ^4 \left(a^2+b^2\right)^2}{2 a b \varepsilon 
   \sqrt{a^4 \left(b^2-\varepsilon ^2\right)^2+2 a^2 b^2 \varepsilon ^2
   \left(b^2+\varepsilon ^2\right)+b^4 \varepsilon ^4}}.
\]
The excess evaluation will be completed by \eqref{ha}:
\[
\sqrt{1+\left(\frac{{\rm d}y}{{\rm d}x}\right)^2}=\frac{\sqrt{\left(a^2+b^2\right)^2x^4 -2 a^2 b^2 \left(a^2-b^2\right)x^2 +a^4 b^4}}{2abx^2}
\]
so that the excess as a function of the abscissa  $\varepsilon$ will be:
\[
\begin{split}
\mathscr{E}(\varepsilon)&=\overline{{\rm PE}}-\overarc[.7]{\rm AE}\\
&=\frac{a^4 b^4-\varepsilon ^4 \left(a^2+b^2\right)^2}{2 a b \varepsilon 
   \sqrt{a^2 b^2-2 a^2 b \varepsilon +a^2 \varepsilon ^2+b^2 \varepsilon ^2}
   \sqrt{a^2 b^2+2 a^2 b \varepsilon +a^2 \varepsilon ^2+b^2 \varepsilon ^2}}\\
   &-\int_{\varepsilon}^{\frac{a b}{\sqrt{a^2+b^2}}}\frac{\sqrt{\left(a^2+b^2\right)^2x^4 -2 a^2 b^2 \left(a^2-b^2\right)x^2 +a^4 b^4}}{2abx^2}\,{\rm d}x.
\end{split}
\] 
\begin{figure}
\begin{center}
\scalebox{.6}{\includegraphics{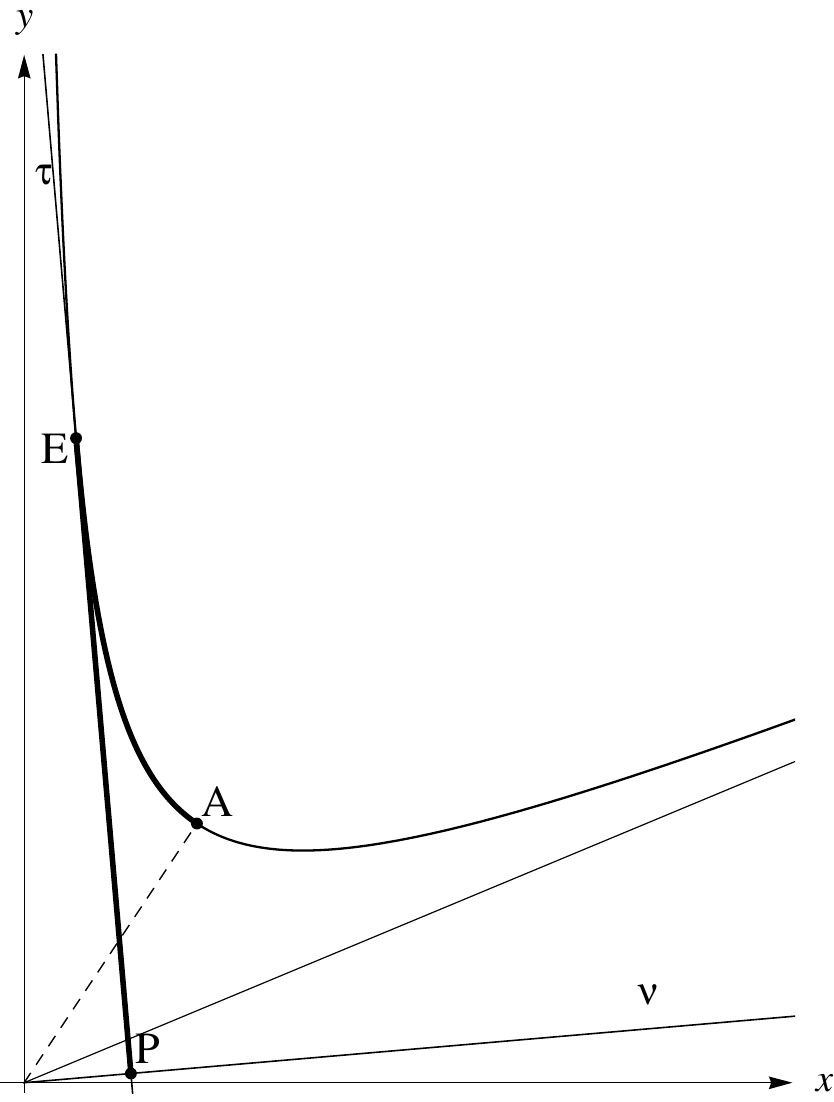}}
\end{center}
\caption{Finite excess =$\overline{{\rm PE}}-\overarc[.7]{\rm AE}$}\label{f24}
\end{figure}
Integrating by parts:
\[
\begin{split}
\overarc[.7]{\rm AE}&=\frac{1}{2ab}\left\{\frac{\sqrt{a^4 b^4+\varepsilon ^4 \left(a^2+b^2\right)^2-2 a^2 b^2 \varepsilon
   ^2 \left(a^2-b^2\right)}}{\varepsilon }-2 a b^2\right.\\
&  \left. +\int_{\varepsilon}^{\frac{a b}{\sqrt{a^2+b^2}}}\frac{2\left(a^2+b^2\right)^2 x^2 -2 a^2 b^2
   \left(a^2-b^2\right)}{\sqrt{\left(a^2+b^2\right)^2x^4 -2 a^2
   b^2 \left(a^2-b^2\right)x^2 +a^4 b^4}}\,{\rm d}x\right\}.
\end{split}
\]
In such a way we get rid of the indetermination, so that, passing to limit for $\varepsilon\to0^{+}$
\begin{equation}\label{cesso:r}
\mathscr{E}(0)=b-\frac{1}{ab}\int_{0}^{\frac{a b}{\sqrt{a^2+b^2}}}\frac{\left(a^2+b^2\right)^2 x^2 - a^2 b^2
   \left(a^2-b^2\right)}{\sqrt{\left(a^2+b^2\right)^2x^4 -2 a^2
   b^2 \left(a^2-b^2\right)x^2 +a^4 b^4}}\,{\rm d}x.
\end{equation}
The excess will be then computed by means of a couple of elliptic integrals:
\begin{align}
\mathscr{E}_1&=\int_{0}^{\frac{a b}{\sqrt{a^2+b^2}}}\frac{\left(a^2+b^2\right)^2 x^2}{\sqrt{\left(a^2+b^2\right)^2x^4 -2 a^2
   b^2 \left(a^2-b^2\right)x^2 +a^4 b^4}}\,{\rm d}x,\\
\mathscr{E}_2&=\int_{0}^{\frac{a b}{\sqrt{a^2+b^2}}}\frac{- a^2 b^2
   \left(a^2-b^2\right)}{\sqrt{\left(a^2+b^2\right)^2x^4 -2 a^2
   b^2 \left(a^2-b^2\right)x^2 +a^4 b^4}}\,{\rm d}x.
\end{align}
Using in sequence \cite{By} entry 361.53 page 215 and 361.53 page 215, we get
\[
\mathscr{E}_1=\frac{ab\sqrt{a^2+b^2}}{2}\left[\boldsymbol{K}\left(\frac{a}{\sqrt{a^2+b^2}}\right)-2\boldsymbol{E}\left(\frac{a}{\sqrt{a^2+b^2}}\right)+\frac{2 b}{\sqrt{a^2+b^2}}\right].
\]
In order to compute $\mathscr{E}_2$ invoking entry 3.138-7 page 259 of \cite{Grr} we get
\[
\mathscr{E}_2=-\frac{a b \left(a^2-b^2\right)}{2 \sqrt{a^2+b^2}}\,\boldsymbol{K}\left(\frac{a}{\sqrt{a^2+b^2}}\right)
\]
Combining $\mathscr{E}_1$ with $\mathscr{E}_2$ we obtain \eqref{cesso:r}.
\end{proof}
\subsection{Comparison with Landen's formula}
Landen shows synthetically the hyperbola of equation
\[
\frac{x^2}{(m-n)^2}-\frac{y^2}{4mn}=1
\]
has an excess given by the difference between the double of quadrantal arc of the ellipse of equation
\[
\frac{x^2}{(m+n)^2}+\frac{y^2}{4mn}=1
\]
and of the quadrantal arc of the ellipse of equation
\[
\frac{x^2}{m^2}+\frac{y^2}{n^2}=1.
\]
This can be done in modern notation by means of complete elliptic integrals:\begin{equation}\label{lcesso}
\mathscr{L}=2m\, \boldsymbol{E}\left(\sqrt{\frac{m^2-n^2}{n^2}}\right)-(m+n)\, \boldsymbol{E}\left(\frac{m-n}{m+n}\right).
\end{equation}
The task will be to prove two determinations of the excess, namely that \eqref{rcesso} and \eqref{lcesso} lead to the same thing. The first step is to express the coefficients $m$ and $n$ by means of the semiaxes $a$ and $b$
\[
\begin{cases}
m-n=a\\
4mn=b^2
\end{cases}
\implies
\begin{cases}
m=\dfrac{1}{2} \left(\sqrt{a^2+b^2}+a\right)\\[2mm]
n=\dfrac{1}{2} \left(\sqrt{a^2+b^2}-a\right)
\end{cases}
\]
Then, by equating \eqref{rcesso} and \eqref{lcesso} we have the identity 
\begin{equation}\label{iid}
\begin{split}
&2\sqrt{a^2+b^2}\,
   \boldsymbol{E}\left(\frac{a}{\sqrt{a^2+b^2}}\right)-\left(\sqrt{a^2+b^2}+a\right)
   \boldsymbol{E}\left(\frac{2 \sqrt{a\sqrt{a^2+b^2}}}{a+\sqrt{a^2+b^2}}\right)=\\
  & \frac{b^2}{\sqrt{a^2+b^2}}\,\boldsymbol{K}\left(\frac{a}{\sqrt{a^2+b^2}}\right).
   \end{split}
\end{equation}
To check that \eqref{iid} is true, we will use theorem 1.2 (c) page 12 of \cite{Borweins} where is proved that:
\begin{equation}\label{Borw}
\boldsymbol{E}(k)=\frac{1+k}{2}\boldsymbol{E}\left(\frac{2\sqrt{k}}{1+k}\right)+\frac{1-k^2}{2}\boldsymbol{K}(k).
\end{equation}
Now, dividing both sides of \eqref{iid} by $\sqrt{a^2+b^2}$ and substituting 
$$
k=\frac{a}{a^2+b^2}
$$
provided that
\[
\frac{2 \sqrt{a\sqrt{a^2+b^2}}}{a+\sqrt{a^2+b^2}}=\frac{2\sqrt{k}}{1+k},\quad \frac{b^2}{a^2+b^2}=1-k^2
\]
which is \eqref{Borw}: namely the excess obtained through the analytical argument of theorem \ref{cess} coincides with Landen's synthetic excess.  
Of course we could change the viewpoint and by two excess representations infer that formula \eqref{Borw} of \cite{Borweins} is true.
 \subsection{Did Landen really create the {\it Landen transformation}?}
 
Landen is often ignored in several historical works. For instance, M. Kline (1908-1992) \cite{Kline}, in all of 1500 pages for 24 centuries of mathematicians, never once mentioned Landen's name. The same with \cite{Katz}, 800 pages. An
exception is G. Loria (1862-1954), \cite{Loriast} who covered  Landen's contribution with 26 lines in total, but where the \lq\lq elliptic'' contribution is misunderstood,
made the common attribution's mistake regarding elliptic integrals, as we will see.
Furthermore Landen is often
quoted carelessly and mostly for the \textit{Residual Analysis}, rather than
for his contributions to the hyperbola or to his (almost unknown) differential solution of algebraic equations. On the contrary, Florian Cajori (1859-1930), \cite{Cajo2}, when describing almost all treatises or mathematics pamphlets, polemics and improvements, during a restricted range of time of the British mathematics, focused for a while on Landen's historical bearing, writing appropriately although not entering analytical details.

Finally, several authors describe Landen's contributions regarding the elliptic integrals, contributions which do not exist. For instance in \cite{Arm} p. 99 formula (4.54) we read:
\begin{quote}
In 1775, Landen gave the formula
\[
\int_{0}^{\phi_1}\left(1-k_{1}^{2}\sin^{2}\theta_{1}\right)^{-1/2}{\rm d}\theta_{1}=(1+k')\int_{0}^{\phi}\left(1-k^{2}\sin^{2}\theta\right)^{-1/2}{\rm d}\theta
\]
where $\sin(\phi_{1})=(1+k')\sin(\phi)cos(\phi)(1-k^{2}\sin^{2}(\phi) )^{-3/2}$
and $k_1=(1-k')/(1+k')$. His proof is to be found in the Philosophical Transactions of the Royal Society, LXV, (1775), page 285.
\end{quote}
and \cite{Mit} had written:

\begin{quote}
Euler's addition theorem and the transformation theory of Landen and
Lagrange were the two fundamental ideas of which the theory of elliptic
function was in possession when this theory was brought up for renoved
consideration by Legendre in 1786.
\end{quote}

Both statements are wrong. Such a transformation was really initiated, unconsciously, by Lagrange, as we shall see. Mistaken references  like these can be found in many treatises and/or articles on elliptic functions: their authors probably did not check on the original Landen sources. Some of them have a better criticism but not always clear ones. A viewpoint we agree with is \cite{Guicc}:
\begin{quote}
Landen's problem was not that of integrating functions of the form
\[
\left[\left(1-x^{2}\right)\left(1-q^{2}x^{2}\right)\right]^{-1/2}
\] 
and he never expressed Landen's transformations in the form known today.
\end{quote}
Accordingly \cite{Cay}, in a rather cautious note writes:

\begin{quote}
Landen's capital discovery is that of the theorem known by his name (obtained in its complete form in the memoir of 1775, and reproduced in the first volume of the Mathematical Memoirs) for the expression of the arc of an hyperbola in terms of two elliptic arcs. To find this, he integrates a differential equation derived from the equation
$$
t=gx \sqrt{\frac{m^{2}-x^{2}}{m^{2}-g^{2}x^{2}}}
$$
interpreting geometrically in an ingenious and elegant manner three integrals which present themselves. If, in the foregoing equation we write $m = 1$, $g = k^{2}$, and instead of $t$ consider the new variable $y = \frac{t}{1-k'}$, then
$$
t=(1+k')x\sqrt{\frac{1-x^{2}}{1-k^{2}x^{2}}}
$$which is the form known as Landen's transformation in the theory of elliptic functions; but his investigation does not lead him to obtain the equivalent of the resulting differential equation
$$
\frac{{\rm d}y}{\sqrt{(1-y^{2})(1-\lambda^{2}y^{2})}}=(1+k')\frac{{\rm d}x}{\sqrt{(1-x^{2})(1-k^{2}x^{2})}}
$$
where	
$$
\lambda=\frac{1-k'}{1+k'}
$$
due it would appear to Legendre and which (over and above Landen's own beautiful result) gives importance to the theorem as leading directly to the quadric transformation of an elliptic integral in regard to the modulus.
\end{quote}
The same caution is shown by some historians like \cite{Cooke} pages 529-539, who affirm the so-called transformation of elliptic integrals to be embedded \textit{inside} the Landen theorem. This untrue, because \textit{referring to} his paper of 1775, there is no match of such a claim; \cite{Greenn} write wisely:
\begin{quote}
While his interest and application in these directions were acute, Landen failed to realize that the whole of his analytical transformations were particular cases of one general one, now known as the Landen transformations.
\end{quote}
\cite{Smadja}, is ascribing the origin of such a distortion to Legendre, who credits Landen with the autorship of transformation, in such a way lessening the merits of Maclaurin, D'Alembert, and of Lagrange too\footnote{Lagrange appreciated Landen, even if he did not cite his work on the hyperbola in the Turin memoria. On the contrary Legendre did not cite Lagrange but Landen, asking himself why Euler never wrote anything on
 Landen whilst, in due time in 1751, had referred to Fagnano. His conclusion is that Euler probably ignored Landen's papers.}. But Landen could not have had the slightest idea of handling those mathematical objects: module, amplitude, and several kinds of irrational integrals which were to be classified by Legendre many years after him. Furthermore Landen could not have had a clear understanding on the relationship between his ellipses as a transformation capable of mapping two sets of coefficients, but keeping an elliptic integral invariant. 
Before the role of Lagrange is shown in the next section, let us try to fix a few points here.

The first to establish a change of variables capable of generating a recursive algorithm in order to compute a particular elliptic integral, was Lagrange in a paper \cite{Lag}, a few years after Landen's. But Legendre, see for instance \textit{Remarque
g\'enerale} page 87 of \cite{Leg1} implemented (ibidem, p. 84) a purely analytical method for rectifying  the hyperbola by means of  elliptic integrals. At some point he introduces
in the integral: 
\[
F(\varphi,k ):=\int_{0}^{\varphi }\frac{\mathrm{d}\theta }{\sqrt{1-k^{2}\sin
^{2}\theta }}
\]
a trigonometric transformation: $\varphi \to \hat{\varphi},\, k\to q$, namely:\footnote{Later, concerning the role of Lagrange, we will guess a possible mathematical path for arriving at this transformation starting from Lagrange's statements. Of course this will show the feasibility of such a path, but nobody could tell which route Legendre took actually.}
\begin{equation}
\sin \left( 2\hat{\varphi} -\varphi \right) =k\sin \varphi,\quad\frac{2\sqrt{k}}{1+k}=q \label{AA2}
\end{equation}
and after this he establishes the relationship:
\begin{equation}
F(\varphi,k)=\frac{2}{1+k}F\left( \hat{\varphi} ,\frac{2\sqrt{k}}{1+k}\right) \label{AA9}
\end{equation}
by Richelot (1808-1875) named \cite{ric}, 
\textit{Legendre Gleichung}, being the new amplitude $\hat{\varphi}$
computed by the old one $\varphi$ by means of \eqref{AA2}.
Nevertheless \eqref{AA9} is referred as \textit{Module Amplitude Transformation}.
About which let us cite again \cite{Cay}:
\begin{quote}
The trigonometrical form [\ldots] does not occur in Landen; it is employed by Legendre, I believe, in an early paper,  \textit{M\'em. de l'Acad. de Paris}, 1786, and in the \textit{Exercices}, 1811, and also in the \textit{Trait\`{e} des Fonctions Elliptiques}, 1825, and by means of it obtains an expression for the arc of hyperbola in terms of two elliptic functions $E(c,\phi),\, E(c',\phi ')$, showing that the arc of the hyperbola is expressible by means of two elliptic arcs, \lq\lq le beau th\'eor\`{e}me dont Landen a enrichi la g\'eom\'etrie''.
\end{quote}
This explains why a \textit {transformation of Lagrange-Legendre} has its own life apart from the \lq\lq theorem of Landen''.
The \cite{Smadja} page 389 recent reconstruction, 
\begin {quote}
Dans le cas qui nous occupe, il est tout aussi l\'egitime de regarder le th\'eor\`{e}me de Landen comme un aboutissement que comme un commencement, selon qu'on l'envisage dans le contexte des recherches de Maclaurin et de d'Alembert qu'il compl\`{e}te et perfectionne, ou dans le contexte correspondant \`{a} la lecture de Legendre. Vouloir absolument trancher, de mani\`{e}re d\'econtextualis\`{e}e, la question de savoir qui de Landen, Lagrange ou Legendre est \lq\lq le premier'' \`{a} d\'egager la transformation de Landen est vain et illusoire
\end{quote}
is the best description of what we deem false.
We take the opportunity to highlight almost all the theorems about elliptical objects attributed to Landen (including e.g. the jacobian Theta functions), by treatises on Calculus like: \cite{WW}; \cite{Bell}; \cite{dur}; \cite{boros}; \cite{Mol}; \cite{Mit}; \cite{San}; \cite{Cas} and so on, are \textit{lacking of any historical value}. As a matter of fact none of formulae referred to in this section was ever written by Landen. A further distortion is provided in \cite{Manna}, page 289 equations (1-17) and (1-18), where Landen is mistakenly credited with deriving some elliptic integrals identities.

\section{Lagrange}
\subsection{The \textit{Nouvelle m\'{e}thode}}
In volume II (1785) of Memoirs of the Turin
Royal Academy we can read Lagrange's long paper 
entitled: \textit{Nouvelle m\'{e}thode de calcul integral pour les
diff\'{e}rentielles affect\'{e}es d'un radical carr\'{e} sous lequel la
variable ne passe pas le quatri\'{e}me degr\'{e}} and where he is concerned with elliptic integrals, namely of rational functions of  $x$ and of the square root of a fourth degree polynomial without multiple roots. He was perfectly aware of the problem's peculiarity:
\begin{quote}
si la plus haute de ces puissances ne monte pas au del\`{a} du
quatri\`{e}me degr\'{e} on peut dans plusieurs cas construire
l'int\'{e}grale par les arcs des s\'{e}ctions coniques.
\end{quote}
which is of poor help in trying to compute them:
\begin{quote}
mais il n'est d'aucune utilit\'{e} pour l'integration effective de
ces differentielles, car la rectification des sections coniques n'est encore
connue que tr\'{e}s-imparfaitement,
\end{quote}
so that the series expansion is, after all:
\begin{quote}
le seul moyen de rappeler \`{a} l'integration toutes les formules
differentielles d'une forme essentiellement irrationelle, 
\end{quote}
whose truncation error can be reduced ad libitum by taking a greater number of the terms expansion. 

Lagrange considers elliptic integrals whose integrand has the form
$$P(x)=M(x)+\frac{N(x)}{\sqrt{a+bx+cx^{2}+ex^{3}+fx^{4}}},$$ 
being $M(x),\,N(x)$ rational functions of $x$, so that the differential to be integrated will be split  in a rational term $M(x)$ integrable through logarithms and arcs of circle, plus an irrational term on which Lagrange concentrated his effort. By means of algebraic transformations he proved
that this term can be split in two terms, rational and irrational, the last being:
\[
\frac{Q(y^{2})}{\sqrt{\varepsilon +\zeta y^{2}+\eta y^{4}}} 
\]
where $\varepsilon ,\zeta ,\eta $ are constants and $Q$ a rational function of  $y^{2}$. 
Next, he showed his method requires only the trinomial $\varepsilon
+\zeta y^{2}+\eta y^{4}$ can be broken in two binomial like $\alpha +\beta 
y^{2},\,\gamma +\delta y^{2}$ being $\alpha,\, \beta,\, \gamma,\,\delta $ real quantities. Now the problem was: how to integrate such a differential. He
succeeded in reducing it to: 
\begin{equation}
\frac{\text{d}y}{\sqrt{\left( 1\pm p^{2}y^{2}\right) \left( 1\pm
q^{2}y^{2}\right) }}  \label{ZZ}
\end{equation}
where $p$ and $q$ are real quantities,  $p>q$, so that the square root is certainly real. Formula \eqref{ZZ} is the starting point  for the arithmetic-geometric transformation that led Lagrange to a particular approach in order to transform the elliptic differentials, but which he did not draw the conclusions from. Our next section will be devoted to this. Let us note briefly that he, from \eqref{ZZ} through successive reductions, arrives at:
\[
\frac{{\rm d}z}{\sqrt{(b^{2}\pm z^{2})^{2}-\beta ^{2}}} 
\]
Expanding the binomial series, and by a term by term integration, he got an infinite series of integrals of rational functions. And then:
\begin{quote}
Est donc assur\'{e} de pouvoir integrer aussi exactement qu'on
voudra toute differentielle affect\'{e} d'un radical carr\'{e} o\`{u} la
variable sous le signe monte jusqu'\`{a} la quatri\`{e}me puissance; ce qui
est le cas d'un grand nombre de prob\`{e}mes g\'{e}ome\'{e}triques et
m\'{e}caniques.
\end{quote}
Such a triumphalism is over: no one, having to evaluate an elliptic integral, with some success would have hoped to complete and
control such a calculation: Lagrange used algebraic techniques to break the 4$^{\rm th}$ degree polynomial in the root, to avoid Ferrari's formul\ae, and complicated root calculations. Lagrange of course realized that the difficult
process should be put
to the test:
\begin{quote}
Comme cette m\'{e}thode est d'un genre assez nouveau, et qu'on
pourrait rencontrer encore quelques difficult\'{e}s dans son usage, nous
allons l'appliquer en d\'{e}tail \`{a} la rectification des arcs elliptiques
et hyperboliques.
\end{quote}
When the eccentricity is very small the elementary elliptic arc can be integrated through a convergent series of even powers of the eccentricity itself, but when it becomes 
\begin{quote}
peu differente
de l'unit\'{e}, ce qui est le cas d'une ellipse ou d'une hyperbole tr\`{e}s
aplatie
\end{quote}
the things become much involved, and even worst if 
\begin{quote}
nous allons appliquer notre m\'{e}thode g\'{e}n\'{e}rale \`{a} la
rectification d'une ellipse e d'une hyperbole quelconque. 
\end{quote}
Lagrange then went on through 30 intricate pages: its complicated approach cannot be expressed by any formula, but by several ones,
holding several parameters often stemming from series expansions. \subsection{The birth of the Arithmetic Geometric Mean}
At this point he constructed a sequence of arithmetic means and a
sequence of geometric means as follows, fix $p=p_0>q=q_0$ and iterate:
\[
p_{n}=\frac{p_{n-1}+q_{n-1}}{2},\quad q_{m}=\sqrt{p_{n-1}q_{n-1}}. 
\]
After this he carried out algebraic arguments concluding that, following the assumptions, the sequence $(p_{n})$ decreases while $(q_{i})$ increases and $q_n<p_n$. The convergence is very fast: a Gauss's example \cite{Gauss} shows that
if for instance $p_0=1$ and $q_0=0.8$, at the third step, the values of $p_{3} 
$ and $q_{3}$ are different since the 12$^{\rm th}$ digit on. Sequences $(
p_{n}) $ and $(q_{n})$ converge to a common limit denoted by $M(p,q)$, say the Arithmetic-Geometric Mean between $p$ and $q$. An extensive treatment of the AGM can be found at \cite{Cox} and \cite{Borweins} and the references therein. Gauss discovered, through the AGM introduced by Lagrange, a way for
computing the elliptic integrals. He provided the arclength $L$ of the
lemniscate, of equation $(x^{2}+y^{2})^{2}=a^{2}(x^{2}-y^{2}),$ as:
\begin{equation}\label{gauslem}
L=4a\int_{0}^{1}\frac{dt}{\sqrt{1-t^{4}}}=\frac{2\pi a}{M\left( 1,\sqrt{2}
\right) } 
\end{equation}
so that nowadays the number $1/M\left(1,\sqrt{2}\right)$ is known as the \textit{Gauss lemniscatic constant}. For Gauss' proof of \eqref{gauslem} see 
\cite{Cox}.

The Arithmetic Geometric Mean was first set forth in the Turin memoir \cite{Lag}
published in 1785\footnote{In 1799 or in 1800 Gauss wrote a paper
(appeared in 1866) describing his many discoveries on the Arithmetic Geometric Mean on which he
had started to work (aged 14) since 1791, as by himself confided to his
friend Schumacher in a letter dated April 16$^{\rm th}$, 1816. On the third volume of his works \cite{Gauss} we can read four entries on the subject.\textit{We are a bit astonished that almost all 
authors ignore the Lagrange's priority,}
assuming Arithmetic Geometric Mean as detected by Gauss. The only valuable exception to this 
wrong
course is due to Almkvist and Berndt \cite{AB}. Unfortunately neither they eluded the other wrong  trend concerning the so-called \lq\lq Landen
transformation'', probably for not having had access to Landen's 1775 paper.}. Afterwords Lagrange proceeded from variable $y$ of \eqref{ZZ},
to $y_{1}$, and from $y_{1}$ to $y_{2}$, and so on:
\[
y=\frac{y_{1}R_{1}}{1\pm q_{1}^{2}y_{1}},\,y_{1}=\frac{y_{2}R_{2}}{1\pm
q_{2}^{2}y_{2}},\,\ldots\,,y_{i}=\frac{y_{i+1}R_{i+1}}{1\pm q_{i+1}^{2}y_{i+1}}, 
\]
being $R_{i}=\sqrt{\left( 1\pm p_{i}^{2}y_{i}^{2}\right) \left( 1\pm
q_{i}^{2}y_{i}^{2}\right) },\quad i=1,\,2,\,\dots,n.$ Next he gave $y_{i+1}^{2}$ as a function of $y_{i}$
\[
y_{i+1}^{2}=\frac{\pm q_{i+1}^{2}y_{i}^{2}-1+R_{i}}{\pm 2p_{i+1}^{2}} 
\]
and established the differential relationships:
\begin{equation}
\frac{dy}{R}=\frac{dy_{1}}{R_{1}}=\frac{dy_{2}}{R_{2}}=\cdots  \label{FF}
\end{equation}
namely: the joint variable transformation on $p$,$q$, and on the variable $y$ of integration kept invariant the elliptic differential \eqref{ZZ}. Notice that Lagrange did not stand much on the differential identities \eqref{FF}
whose  genesis is not explained at all. Soon after he went on other subjects. Exactly at this point we can show that, starting from \eqref{FF} one arrives at the famous Legendre
identity.

\subsection{From Lagrange differential identity to Legendre's formula}
This section is devoted to fill a gap: Lagrange did not complete his work but Legendre somehow succeeded in providing the famous transformation. What could have happened in the meantime? We do not really know, but we succeeded in filling such a gap by introducing what Legendre \textit{could have done}. 
Given $0<q<p,$ let us start by analyzing in  \eqref{ZZ} the case with negative signs
\[
\int \frac{\mathrm{d}y}{\sqrt{(1-p^{2}y^{2})(1-q^{2}y^{2})}}. 
\]
The positivity ranges for the expression under root, where one wishes to work, lead to 
\[
(1-p^{2}y^{2})(1-q^{2}y^{2})\geq 0\iff y\in \left( -\infty ,-1/q\right] \cup
\left[ -1/p,1/p\right] \cup \left[ 1/q,\infty \right) ; 
\]
and then we are going to consider the definite integral 
\begin{equation}
I(p,q):=\int_{0}^{1/p}\frac{\mathrm{d}y}{\sqrt{(1-p^{2}y^{2})(1-q^{2}y^{2})}}
\label{defmeno}
\end{equation}
by introducing functions' family
\[
R_{p,q}(y):=\sqrt{(1-p^{2}y^{2})(1-q^{2}y^{2})} 
\]
and $p_{1}<p$ and  $q_{1}>q$ with $p_{1}>q_{1},$ a new integration variable
$y_{1}$ in function of $y$ is introduced: 
\[
y=\frac{y_{1}}{1-q_{1}^{2}\,y_{1}^{2}}\,R_{p_{1},q_{1}}(y_{1}), 
\]
or: 
\begin{equation}
y=y_{1}\sqrt{\frac{1-p_{1}^{2}y_{1}^{2}}{1-q_{1}^{2}y_{1}^{2}}}
\label{stremo}
\end{equation}
Let us work inside the set where $1-p_{1}^{2}y_{1}^{2}>0$ and  $
1-q_{1}^{2}y_{1}^{2}>0$. Converting the former differential in terms of the new variable $y_{1}$%
\[
\mathrm{d}y=\frac{p_{1}^{2}q_{1}^{2}y_{1}^{4}-2p_{1}^{2}y_{1}^{2}+1}{\left(
1-q_{1}^{2}y_{1}^{2}\right) \sqrt{\left( 1-p_{1}^{2}y_{1}^{2}\right) \left(
1-q_{1}^{2}y_{1}^{2}\right) }}\,\mathrm{d}y_{1} 
\]
so that  the integral \eqref{defmeno} is changed in: 
\[
\int_{0}^{s(p,q)}\frac{p_{1}^{2}q_{1}^{2}y_{1}^{4}-2p_{1}^{2}y_{1}^{2}+1}{%
\sqrt{\left( 1-p_{1}^{2}y_{1}^{2}\right) \left( 1-q_{1}^{2}y_{1}^{2}\right)
\left( p^{2}p_{1}^{2}y_{1}^{4}-(p^{2}+q_{1}^{2})y_{1}^{2}+1\right) \left(
p_{1}^{2}q^{2}y_{1}^{4}-(q^{2}+q_{1}^{2})y_{1}^{2}+1\right) }}\,\,\mathrm{d}%
y_{1} .
\]
Now, assuming $p_{1}$ and $q_{1}$ to come by the above iteration of arithmetic-geometric type,
\[
\begin{cases}
p_1=\dfrac{p+q}{2}\\
q_1=\sqrt{pq}
\end{cases}
\iff 
\begin{cases}
p=p_1+\sqrt{p_1^2-q_1^2}\\
q=p_1-\sqrt{p_1^2-q_1^2}
\end{cases}
\]
we succeed in simplifying
\[
\begin{split}
&\int_{0}^{s(p,q)}\frac{p_{1}^{2}q_{1}^{2}y_{1}^{4}-2p_{1}^{2}y_{1}^{2}+1}{%
\sqrt{\left( 1-p_{1}^{2}y_{1}^{2}\right) \left( 1-q_{1}^{2}y_{1}^{2}\right)
\left( p_{1}^{2}q_{1}^{2}y_{1}^{4}-2p_{1}^{2}y_{1}^{2}+1\right) ^{2}}}\,\,%
\mathrm{d}y_{1}=\\
&\int_{0}^{s(p,q)}\frac{\mathrm{d}y_{1}}{\sqrt{\left(
1-p_{1}^{2}y_{1}^{2}\right) \left( 1-q_{1}^{2}y_{1}^{2}\right) }} 
\end{split}
\]
where the polynomial $
p_{1}^{2}q_{1}^{2}y_{1}^{4}-2p_{1}^{2}y_{1}^{2}+1$ results to be positive for $%
y_{1}\in [0,s(p,q)].$ In such a way we obtained in integral form Lagrange's differential identity \eqref{FF}. The upper integration limit $s(p,q)$ is obtained
by putting $y=p$ in \eqref{stremo}, providing $p_{1}$ and $q_{1}$ as a function of $p$ and $q$ 
(Arithmetic Geometric Mean) and solving to $y_{1}$: 
\[
s(p,q)=\sqrt{\frac{2}{p(p+q)}}=\frac{1}{\sqrt{p_{1}\sqrt{p_{1}^{2}-q_{1}^{2}}%
+p_{1}^{2}}} .
\]
The case of positive signs is analogous, and really more simple, not needing a signs discussion. 
Being $0<q<p$, taking $0<x<1/p$, it is true that
\begin{equation}
\int_{0}^{x}\frac{\mathrm{d}y}{\sqrt{(1-p^{2}y^{2})(1-q^{2}y^{2})}}%
=\int_{0}^{s(x,p,q)}\frac{\mathrm{d}y_{1}}{\sqrt{%
(1-p_{1}^{2}y_{1}^{2})(1-q_{1}^{2}y_{1}^{2})}}  \label{glei}
\end{equation}
with
\[
s(x,p,q)=\frac{\sqrt{2}}{p+q}\,\sqrt{1+pqx^{2}-\sqrt{\left(
1-p^{2}x^{2}\right) \left( 1-q^{2}x^{2}\right) }}
\]
and if $x=1/p$ then $s(1/p,p,q)=s(p,q)$.
Notice that
the expression for $s(p,q)$  is more simple than $s(x,p,q)$
being the value $y=1/p$ the absolute maximum attained by the function used by Lagrange in \eqref{stremo}: 
\[
y=y_{1}\sqrt{\frac{1-p_{1}^{2}y_{1}^{2}}{1-q_{1}^{2}y_{1}^{2}}}=y_{1}\sqrt{%
\frac{1-\frac{1}{4}y_{1}^{2}(p+q)^{2}}{1-pqy_{1}^{2}}}:=f(y_1).
\]
Noting that
\[
\frac{q_{1}}{p_{1}}=\frac{2\sqrt{pq}}{p+q}=\frac{2\sqrt{k}}{1+k}:=\hat{k},\quad
\frac{p}{p_{1}}=\frac{2}{1+k} 
\]
putting $k:=q/p$ and making use of homothetics $z=p\,y,\,z_{1}=p_{1}\,y_{1}$ at both sides of \eqref{glei}, we get:
\begin{equation}
\int_{0}^{px}\frac{\mathrm{d}z}{\sqrt{(1-z^{2})(1-k^{2}z^{2})}}=\frac{2}{1+k}%
\int_{0}^{p_{1}\,s(x,p,q)}\frac{\mathrm{d}z}{\sqrt{(1-z_{1}^{2})(1-\hat{k}
^{2}z_{1}^{2})}}  \label{gleichung}
\end{equation}
published for the first time on at page 7 of the first volume in his
{\it Trait\'{e}} and by Jacobi and Richelot named as \textit{Legendre Gleichung}. This fact, never observed before, strengthens the idea one shall refer to a \lq\lq Lagrange-Legendre transformation'' apart from the hyperbolic \lq\lq Landen theorem''. Let us see the relationship between the amplitudes in two consecutive stages of iteration. For that of the integral at left hand side of \eqref{gleichung} one finds:
 $\sin \varphi =px,$ and
\begin{equation}
\sin \hat{\varphi}=p_{1}\,s(x,p,q)=\sqrt{\frac{1+pqx^{2}-\sqrt{\left(
1-p^{2}x^{2}\right) \left( 1-q^{2}x^{2}\right) }}{2}}  \label{hat}
\end{equation}
defines that of the integral at right hand side of  in \eqref{gleichung}. The wanted link between $\varphi $ and $\hat{\varphi}$ is provided by:

\begin{teorema}
The Lagrange's AGM trasformation on a first kind elliptic integral $F(k,\varphi )$ changes its parameters $k,\varphi $,
defining a new amplitude $\hat{\varphi}$ such that\footnote{Notice that the relationship between the amplitudes ${\varphi}$ and $\hat{\varphi}$  as in the first of \eqref{AA2} is the same thing as \eqref{landen}.}: 
\begin{equation}
\tan \varphi =\frac{\sin (2\hat{\varphi})}{k+\cos (2\hat{\varphi})} .
\label{landen}
\end{equation}
\end{teorema}
\begin{proof} It is well known that $\varphi $ and $\hat{\varphi}$
are ranged between 0 and $\pi /2$ and then 
\[
\tan \varphi =\frac{\sin \varphi }{\sqrt{1-\sin ^{2}\varphi }}=\frac{px}{%
\sqrt{1-p^{2}x^{2}}}. 
\]
On the other side
\[
\frac{\sin (2\hat{\varphi})}{k+\cos (2\hat{\varphi})}=\frac{2\sin \hat{%
\varphi}\sqrt{1-\sin ^{2}\hat{\varphi}}}{k+1-2\sin ^{2}\hat{\varphi}} 
\]
Thesis follows by the fact that, using \eqref{hat}, reminding that  $k=q/p$
one can write: 
\[
\begin{split}
\frac{\sin (2\hat{\varphi})}{k+\cos (2\hat{\varphi})}&=\frac{p\sqrt{%
x^{2}\left( p^{2}+q^{2}-2p^{2}q^{2}x^{2}+2pq\sqrt{\left( 1-p^{2}x^{2}\right)
\left( 1-q^{2}x^{2}\right) }\right) }}{q-p^{2}qx^{2}+p\sqrt{\left(
1-p^{2}x^{2}\right) \left( 1-q^{2}x^{2}\right) }}\\
&=\frac{px}{\sqrt{
1-p^{2}x^{2}}}. 
\end{split}
\]
\end{proof} 
The well-known link \eqref{landen} between amplitudes has then be proved starting from the Lagrange's arithmetic-geometric transformation.

Lagrange appreciated Landen's work (see: \textit{Th\'{e}orie des
fonctions analytiques}, 1797, about the \textit{Residual Analysis,} 1758). On the subject of hyperbolic Landen theorem, in his letter to Condorcet\footnote{Jean-Antoine-Nicolas de Caritat, marquis De Condorcet (1743-1794), a French mathematician and economist.} of January 3$^{\rm rd}$, 1777, see \cite{Oplag}, tome XIV, page 41, we can read:
\begin{quote}
J'ai vu, dans le dernier volume des Transactions philosophiques, un th\'{e}or\`{e}me de M. Landen qui me parait bien singulier. Il r\'{e}duit la rectification des arcs elliptiques \`{a} celle des arcs hyperboliques. Je n'ai pas encore eu le temps d'examiner s'il n'y a pas de paralogisme dans la d\'{e}monstration.
\end{quote}
Nevertheless, given that the elliptic arc was as a primary element for computing even more difficult fluents, the true meaning of the new is the reverse, namely the reduction of one of the most difficult (hyperbola) integrations, to some more simple objects, like the elliptic arcs.
However Lagrange did not quote Landen's articles on the hyperbola in his Turin memoir \cite{Lag}, whilst Legendre did not cite Lagrange at all, but mentioned Landen, asking himself why Euler did not write anything about Landen whilst in his time (1751) he had written about Fagnano. He concludes that most likely Euler didn't know Landen papers. 
Formula \eqref{landen} was unknown to Landen; where Legendre took it, we cannot say. Jacobi, as we will see in a next paper, stayed on the subject, providing a geometric interpretation. Furthermore, not Gauss, but Lagrange, author of that paper only on the elliptic integrals, established the Arithmetic-Geometric Mean using it for building a transformation of the basic values $p$ and $q$ which. Working on it one can advance the formula:
\[
\frac{2\sqrt{k}}{1+k}:=\hat{k} 
\]
for scaling modules; and that (\ref{landen}) for amplitudes, both due to Legendre\footnote{%
For being more precise the amplitudes scaling (\ref{landen})
is introduced and used by Legendre in the form: 
\[
\sin(2\varphi _{1}-\varphi )=k\sin \varphi 
\]
equivalent to the previous one.}, who published them in first volume
of his {\it Trait\`{e}} submitted to the Academy in 1825 and issued in 1827.

\section{Conclusions}

This paper's aim is the hyperbola rectification, 1742-1827, with all the relevant problems of analytical, historical, geometrical nature.
Our conclusions have been split up to give the due room separately to each of the founders of the theory, Maclaurin, Landen, Lagrange.
The Legendre contribution will be the object of a next specific treatment.

\subsection*{Maclaurin}
 
 In \textit{Fluxions}, n. 755, Maclaurin, moved by his interest in all Mathematical Physics and Calculus of his time, defined a research program concerning the classification of irrational fluents,  then followed by D'Alembert  too, \textit{Recherches sur le calcul integral}, through a purely analytical approach. On the contrary Maclaurin performed their integration by means of arcs of conics and often with the help of geometrical arguments.
 Among these problems the \textit{elastica} was absolutely crucial. Not being possible to integrate it, such a problem can be switched either in rectifying the lemniscate or in computing the hyperbolic excess, at which then he arrived from the elastica and strain analysis.
His Apollonius's tributary scheme for measuring a hyperbola's length is analytically explained and discussed. Finally, his procedure for evaluating the hyperbolic excess is provided, adding to it our modern treatment and comparing the results. 

\subsection*{Landen}

We try to make known his famous theorem on hyperbola rectification whose original, rather obscure proofs (1775, 1780) are far from easy. After detailed explanations going back to Maclaurin's previous fundamental analysis and so on, we add some geometrical interpretations to all the Landen processes by means of a continuous chain of geometry constructions inferred by his treatment. We show that he studied the same irrational fluent of Maclaurin, namely the hyperbolic excess, which for him did not stem from the elastica, but from the pendulum time equation.
His synthetical topic on the limit hyperbolic excess has been analytically confirmed by us through the elliptic functions which will become a standard much later. Finally,  Landen is proved to be completely irrelevant to the transformation bearing his name, and that his name applies only to the hyperbolic theorem.

\subsection*{Lagrange}

In the \textit{Nouvelle m\'{e}thode}, \cite{Lag} defines a 
differential identity stemming from the AGM, established by him in the same paper and mistakenly attributed to Gauss. Integrating his identity, we arrive at the well-known Legendre formula for a recursive computation of the first kind elliptic integral. Such a  transformation was completely \textit{unknown to Landen}, as one can understand from \cite{lan1}, \cite{lan2} and \cite{lan3}: it was envisioned briefly but not developed by Lagrange, who was ahead of his times, but not very interested in elliptic integrals. The transformation was published in 1827 by Legendre, who applied  it extensively throughout the first volume of his {\it Trait\'{e}
}. 

Giovanni Mingari Scarpello

via Negroli, 6 Milan, Italy

giovannimingari@yahoo.it

\hspace{0.3 cm}

Daniele Ritelli

Dipartimento di Statistica University of Bologna

viale Filopanti, 5 Bologna, Italy

daniele.ritelli@unibo.it

\hspace{0.3 cm}

Aldo Scimone

Via C. Nigra, 30 Palermo, Italy

aldo.scimone@libero.it

\end{document}